\documentclass[12pt]{amsart}
\usepackage[utf8]{inputenc}
\usepackage[T1]{fontenc}
\usepackage{amsxtra,amssymb,amsthm,amsmath,amscd,mathrsfs, epsfig, eufrak}
\usepackage{amscd, amsmath, mathrsfs, amssymb, amsthm, amsxtra, bbding, epsfig, eucal, eufrak, graphicx, latexsym, mathrsfs, url, color, mathbbol, bbold}
\usepackage[all]{xy}
\usepackage{bbding}
\usepackage{float}
\usepackage{graphicx}
\usepackage{latexsym}
\usepackage{epsfig,epstopdf,color}
\usepackage{bbold}
\usepackage{bbding}
\usepackage{oldgerm}
\usepackage[french]{babel}
\usepackage{titletoc}

\theoremstyle{plain}
\newtheorem{Theorem}{Th\'{e}or\`{e}me}
\newtheorem*{Theorem*}{Theorem}
\newtheorem{Lemma}{Lemme}[section]
\newtheorem*{Hypothese*}{Hypoth\`{e}se}

\newtheorem{corollaire}{Corollaire}
\newtheorem*{corollaire*}{Corollaire 2*}
\newtheorem{proposition}{Proposition}

\theoremstyle{definition}

\newtheorem*{Remark}{Remarque}

\newtheorem*{remerciements}{Remerciements}

\numberwithin{equation}{section}

\usepackage{tikz}
\usepackage[normalem]{ulem}
\usepackage{soul}
\usepackage{color}
\setstcolor{red}

\def\le{\leqslant}
\def\ge{\geqslant}

\begin{document}

\title[Sur les plus grands facteurs premiers d'entiers consécutifs]
{Sur les plus grands facteurs premiers d'entiers consécutifs}
\date{\today}
\author{Zhiwei Wang \ (Nancy)}

\address{%
Institut \'Elie Cartan de Lorraine\\
Universit\'e de Lorraine\\
UMR 7502\\
54506 Van\-d\oe uvre-l\`es-Nancy\\
France
}
\email{zhiwei.wang@univ-lorraine.fr}

\thanks{L'auteur est partiellement soutenu par une bourse de ``China Scholarship Council''.}

\begin{abstract}
Let $P^+(n)$ denote the largest prime factor of the integer $n$ and $P_y^+(n)$ denote the largest prime factor $p$ of $n$ which satisfies $p\leqslant y$. In this paper, firstly we show that the triple consecutive integers with the two patterns $P^+(n-1)>P^+(n)<P^+(n+1)$ and $P^+(n-1)<P^+(n)>P^+(n+1)$ have a positive proportion respectively. More generally, with the same methods we can prove that for any
$J\in \mathbb{Z}, J\geqslant3$, the $J-$tuple consecutive integers with the two patterns
$P^+(n+j_0)= \min\limits_{0\leqslant j\leqslant J-1}P^+(n+j)$ and
$P^+(n+j_0)= \max\limits_{0\leqslant j\leqslant J-1}P^+(n+j)$ also have a positive proportion respectively. Secondly for $y=x^{\theta}$ with $0<\theta\leqslant 1$ we show that there exists a positive proportion of integers $n$ such that $P_y^+(n)<P_y^+(n+1)$.
Specially, we can prove that the proportion of integers $n$ such that $P^+(n)<P^+(n+1)$ is larger than 0.1356, which improves the previous result ``0.1063'' of the author.
\end{abstract}
\vglue -1,5mm
\maketitle
{\footnotesize
\tableofcontents

\dottedcontents{section}[1.16cm]{}{1.8em}{5pt}
\dottedcontents{subsection}[2.00cm]{}{2.7em}{5pt}
}

\section{Introduction}

Les entiers naturels ont deux structures fondamentales: additive et multiplicative.
En général, les propriétés multiplicatives d'un entier et celles de sa perturbation additive sont indépendantes.
Les nombres premiers de Fermat et les nombres premiers jumeaux sont deux exemples typiques.
Dans cet article, nous nous intéressons aux facteurs premiers des entiers consécutifs.
D\'{e}signons par $P^+(n)$ le plus grand facteur premier d'un entier g\'{e}n\'{e}rique $n\geqslant 1$ avec la convention que $P^+(1)=1$.
En tenant compte de la raison mentionnée ci dessus, il est naturel d'escompter que $P^+(n)<P^+(n+1)$ pour un entier sur deux et plus généralement:
\begin{Hypothese*}[\textbf{A}]
Soit $k\geqslant 2$ un entier fix\'{e}.
Alors pour toute permutation $(a_1, a_2, \ldots , a_k)$ de $\{0, 1, \ldots, k-1\}$,  on a
$$
{\rm{Prob}}[P^+(n+a_1)<P^+(n+a_2)<\cdots <P^+(n+a_k)]=\frac{1}{k!},
$$
c'est-à-dire,
\begin{equation}\label{Conjecture}
\frac{1}{x} \sum_{\substack{n\le x\\ P^+(n+a_1)<P^+(n+a_2)<\cdots <P^+(n+a_k)}} 1
\to \frac{1}{k!}
\end{equation}
pour $x\to\infty$.
\end{Hypothese*}

Cette conjecture est formul\'{e}e par De Koninck et Doyon \cite{DeKo11} dans le cadre de leur article sur
la distance entre les entiers friables.
Sans doute, une telle conjecture est très difficile à démontrer.
M\^{e}me dans le cas le plus simple, i.e. $k=2$, cette conjecture reste encore ouverte.
Ce cas est un des ``unconventional problems in number theory'' d'Erd\H{o}s (voir par exemple \cite{Erd79} ou \cite{Ten13}).

\vskip 1mm

\subsection{Les plus grands facteurs premiers de trois entiers cons\'{e}cutifs}\

\vskip 1mm

En 1978, Erd\H{o}s et Pomerance observent  dans leur article {\cite{ErdPom78}} que les deux configurations
\begin{equation}\label{cas1}
P^+(n-1)>P^+(n)<P^+(n+1)
\end{equation}
ou
\begin{equation}\label{cas2}
P^+(n-1)<P^+(n)>P^+(n+1)
\end{equation}
ont lieu pour une infinit\'{e} d'entiers $n$, et conjecturent qu'elles se produisent pour une proportion positive d'entiers. Par ailleurs, ils d\'{e}montrent l'existence d'une infinit\'{e} d'entiers $n$ satisfaisant
\begin{equation}\label{cas3}
P^+(n-1)<P^+(n)<P^+(n+1)
\end{equation}
en considérant des entiers $n$ de la forme $n=p^{k_0}$ avec $k_0$ judicieusement choisi.
Finalement, ils remarquent que \og On the other hand we cannot
find infinitely many $n$ for which
\begin{equation}\label{cas4}
P^+(n-1)>P^+(n)>P^+(n+1),
\end{equation}
but perhaps we overlook a simple proof.\fg\;
En 2001, Balog \cite{Bal01} d\'{e}montre leur conjecture et il obtient
\begin{align*}
\big|\big\{n\leqslant x : P^+(n-1)>P^+(n)>P^+(n+1)\big\}\big|\gg x^{1/2}
\end{align*}
pour $x\rightarrow \infty$.

Dans cet article, nous montrons qu'il existe une proportion positive d'entiers $n$ tels que \eqref{cas1} et \eqref{cas2} sont vraies.

Notre résultat est le suivant.

\begin{Theorem}\label{thm1}
Pour $x\rightarrow \infty$, on a
\begin{align}\label{eq:thm1_A}
\big|\big\{n\leqslant x:\, P^+(n-1)>P^+(n)<P^+(n+1)\big\}\big|> 1,063\times 10^{-7}x.
\end{align}
et
\begin{align}\label{eq:thm1_B}
\big|\big\{n\leqslant x:\, P^+(n-1)<P^+(n)>P^+(n+1)\big\}\big|> 8,84\times 10^{-4}x.
\end{align}
\end{Theorem}

L'idée de la preuve est de considérer pour \eqref{eq:thm1_A} des entiers $n$ friables et pour \eqref{eq:thm1_B} des
entiers de la forme $mp$, où $p$ est un nombre premier de taille assez grande. On introduit ensuite un système de poids
bien adapté et on utilise des théorèmes de type Bombieri-Vinogradov pour les entiers friables et pour les entiers avec un grand facteur premier.

\vskip 1mm

\subsection{Les plus grands facteurs premiers de deux entiers cons\'{e}cutifs}\

\vskip 1mm

Dans ce sous-paragraphe,
nous considérons les plus grands facteurs premiers de deux entiers consécutifs.
Dans ce cas, la conjecture \eqref{Conjecture} peut être simplifiée de la manière suivante :
\begin{equation}\label{conj:2}
|\{n\le x : P^+(n)<P^+(n+1)\}|\sim \tfrac{1}{2} x.
\end{equation}
En 1978, Erd\H{o}s et Pomerance {\cite{ErdPom78}} d\'{e}montrent qu'il existe une proportion positive d'entiers $n$ avec $P^+(n)<P^+(n+1)$. Plus pr\'{e}cis\'{e}ment, ils obtiennent
\begin{equation}\label{0.0099}
\big|\big\{n\leqslant x:\, P^+(n)<P^+(n+1)\big\}\big|> 0,0099x \qquad (x\rightarrow \infty).
\end{equation}
En 2005, La Bret\`{e}che, Pomerance et Tenenbaum {\cite{DelaPomTen05}} am\'{e}liorent la constante 0,0099 en 0,05544. De plus, dans leur article ils indiquent que 0,05544 peut être remplacée par 0,05866, gr\^{a}ce \`{a} une observation de Fouvry.
Récemment, nous \cite{Wang17} avons réussi à généraliser ce problème dans les petits intervalles.
En particulier, nous avons amélioré la constante 0,05866 en 0,1063.

Rivat \cite{Rivat01} propose une autre voie pour approcher \eqref{conj:2}. Pour $2\leqslant y\leqslant x,$
notons $P_y^+(n)$, le plus grand facteur premier de $n$ inférieur à $y$, $P_y^+(n)=\max\{p|n:\, p\leqslant y\},$
avec la convention $P_y^+(n)=1$ si le plus petit facteur premier de $n$ est strictement supérieur à $y$.

Pour quels $y$, avons-nous la formule
\begin{equation}\label{Rivia:1}
|\{n\le x : P^+_y(n)<P^+_y(n+1)\}|\sim \tfrac{1}{2} x\, ?
\end{equation}

En posant
$$
f_y(n) := \begin{cases}
1 & \text{si $\,P^+_y(n+1)>P^+_y(n)$,}
\\\noalign{\vskip 1mm}
-1 &  \text{si $\,P^+_y(n+1)<P^+_y(n)$,}
\end{cases}
$$
Rivat montre que
$$
\Big|\sum_{\substack{1\leqslant an+b\leqslant x}}f_y(an+b)\Big|
\ll_{a, b} x\exp\bigg(-\frac{\log x}{10\log y}\bigg)
$$
est valable pour  $\big(a,\, b(b+1)\big)=1$ et
\begin{align}\label{y-Rivat}
x\ge 3
\qquad\text{et}\qquad
3\leqslant y\leqslant \exp\bigg(\frac{\log x}{100\log_2x}\bigg).
\end{align}
Cela implique que la formule asymptotique \eqref{Rivia:1} a lieu
uniformément dans le domaine \eqref{y-Rivat}.
Puisque \eqref{conj:2} est équivalente à \eqref{Rivia:1} avec $y=x$,
il serait intéressant d'étendre le domaine de $y$ dans \eqref{y-Rivat} ci-dessus \`{a} $y=x^{\alpha}$
o\`{u} $0<\alpha\leqslant 1$ est une constante.

Nous n'avons pas réussi à obtenir une telle extension de \eqref{y-Rivat}. Notre théorème fournit cependant une
minoration du membre de gauche de \eqref{Rivia:1} quand $y=x^{\alpha}$.

\begin{Theorem}\label{thm2}
Soit $\alpha\in ]0,\, 1].$ Il existe $C(\alpha)>0$ tel que
\begin{align}
\big|\big\{n\leqslant x:\ P^+_y(n)<P^+_y(n+1)\big\}\big|
\geqslant C(\alpha)x
\end{align}
pour $x\rightarrow \infty$.
\end{Theorem}
Notre démonstration fournit des expressions explicites de $C(\alpha)$ admissibles.
On renvoie le lecteur à \eqref{def:C(alpha)(i)} et \eqref{def:C(alpha)(ii)} pour une définition précise de
$C(\alpha)$ respectivement dans les intervalles $]0, 1/2]$ et $]1/2, 1].$ Le suivant est trois exemples de valeurs de $C(\alpha)$ admissibles:
\begin{align*}
C\Big(\frac{1}{3}\Big)>0,0506,\qquad C\Big(\frac{1}{2}\Big)>0,0914,\qquad C\Big(\frac{2}{3}\Big)>0,0948.
\end{align*}
Lorsque $\alpha$ tend vers $0$, notre résultat devient moins intéressant, en effet,
$$\lim\limits_{\alpha\rightarrow 0}C(\alpha)\rightarrow 0,$$
ce qui est contraire à l'intuition. Cela est d\^{u} au système de poids que nous utilisons qui est certainement très
perfectible pour des petites valeurs de $\alpha.$

Dans cet article nous souhaitons donner un premier résultat valable pour tout $\alpha\in ]0, 1].$ Nous espérons dans
un prochain travail améliorer les valeurs de $C(\alpha)$ pour $\alpha$ proche de $0.$

Pour $\alpha=1$, le Théorème \ref{thm2} fournit une amélioration de la proportion $0,1063$
obtenue dans \cite{Wang17}. Nous montrons ainsi que $P^+(n)<P^+(n+1)$ (ou $P^+(n)>P^+(n+1)$)
a lieu pour au moins 2 entiers sur 15, plus précisément on a le corollaire suivant.

\begin{corollaire}\label{cor1}
Pour $x\rightarrow \infty$, on a
\begin{align}\label{eq:cor1}
\big|\big\{n\leqslant x:\, P^+(n)<P^+(n+1)\big\}\big|> 0,1356x.
\end{align}
\end{corollaire}

Signalons que la constante  0,1356 peut être remplacée par 0,411 sous l'hypothèse d'Elliott-Halberstam
et l'hypothèse d'Elliott-Halberstam pour des entiers friables (avec ``$q\leqslant x^{1-\varepsilon}\,$'' à la place de ``$q\leqslant x^{1/2}/(\log x)^B\,$'' dans \eqref{eq:BV-S(x,y)} du Lemme \ref{lem:BV-S(x,y)} ci-dessous).

\vskip 1mm

\subsection{Les plus grands facteurs premiers de plusieurs entiers cons\'{e}cutifs}\

\vskip 1mm

De Koninck et Doyon \cite{DeKo11} ont remarqué que
sous l'Hypothèse (A) on a
$$
\big|\big\{n\leqslant x:\, P^+(n+j_0)= \min_{0\leqslant j\leqslant J-1}P^+(n+j)\big\}\big|
\sim J^{-1}x
$$
pour $x\to\infty$.
Notre méthode permet d'obtenir une proportion positive inconditionnelle.

\begin{Theorem}\label{thm3}
Soient $J\geqslant 3$ un entier et $j_0\in \{0, \dots, J-1\}$.
Alors on a
\begin{align}\label{eq:thm3}
\big|\big\{n\leqslant x : P^+(n+j_0)= \min_{0\leqslant j\leqslant J-1}P^+(n+j)\big\}\big|
\ge \{C_3(J)+o(1)\} x
\end{align}
pour $x\rightarrow \infty$, o\`{u}
\begin{align}\label{C3(J)}
C_3(J)
:= \max_{0<\alpha<\frac{1}{2(J-1)}} \rho\bigg(\frac{1}{\alpha}\bigg)
\bigg(\alpha\log\frac{1}{2\alpha(J-1)}\bigg)^{J-1}>0
\end{align}
\end{Theorem}

Nous avons un résultat similaire pour le max à la place de min.

\begin{Theorem}\label{thm4}
Soient $J\geqslant 3$ un entier et $j_0\in \{0, \dots, J-1\}$.
Alors on a
\begin{align}\label{eq:thm4}
\big|\big\{n\leqslant x : P^+(n+j_0) = \max_{0\leqslant j\leqslant J-1}P^+(n+j)\big\}\big|
\ge \{C_4(J)+o(1)\} x
\end{align}
pour $x\rightarrow \infty$, o\`{u}
\begin{align}\label{C4(J)}
C_4(J)
:= \max_{\substack{\frac{2J-2}{2J-1}<\alpha<1 \\ 1-\alpha\leqslant \beta< \gamma<\frac{\alpha}{2(J-1)}}}
\bigg(\beta\log\frac{\gamma}{\beta}\bigg)^{J-1}\log\frac{1}{\alpha}>0.
\end{align}
\end{Theorem}

\subsection{Application : distance entre les entiers friables}\

\vskip 1mm

Pour mesurer la distance entre les entiers friables,
De Koninck et Doyon \cite{DeKo11} ont introduit la fonction suivante :
\begin{align}\label{delta(n)}
\delta(n) := \min_{\substack{m\in \mathbb{N}^*\setminus\{n\}\\ P^+(m)\leqslant P^+(n)}}|m-n|
\end{align}
et ont montré sous l'Hypothèse (A) la formule asymptotique
$$
\sum_{n\leqslant x} \delta(n)^{-1}\sim (4\log 2-2)x
$$
pour $x\rightarrow \infty$.
De plus, le Th\'{e}or\`{e}me 10 de \cite{DeKo11} entra\^{i}ne la minoration inconditionnelle
\begin{align}\label{mino-delta(n)}
\sum_{n\leqslant x} \delta(n)^{-1} >\tfrac{2}{3}x+o(x)
\end{align}
pour $x\rightarrow \infty$.

Nous pouvons obtenir gr\^{a}ce au Théorème \ref{thm1}, une majoration inconditionnelle.

\begin{corollaire}\label{cor2}
Pour $x\rightarrow \infty$, on a
\begin{align}\label{eq:cor2}
\sum_{n\leqslant x} \delta(n)^{-1} <(1-5,315\times 10^{-8})x .
\end{align}
\end{corollaire}
On remarque qu'on peut très légèrement améliorer ce résultat (voir \eqref{cor2-thm3})
en appliquant le Théorème \ref{thm3} à la place du Théorème \ref{thm1}.

Par analogie à $\delta(n)$, nous proposons étudier la zone de $P(n)-$friabilité autour de $n$
\begin{align*}
\delta_*(n) := \min_{\substack{m\in \mathbb{N}^*\setminus\{n\}\\ P^+(m)\geqslant P^+(n)}}|n-m|.
\end{align*}
En adaptant la d\'{e}monstration du Th\'{e}or\`{e}me 10 de \cite{DeKo11}, on peut facilement montrer que pour $x\rightarrow \infty$
\begin{align}\label{mino-delta*(n)}
\sum_{n\leqslant x} \delta_*(n)^{-1}> \big\{\tfrac{2}{3}+o(1)\big\}x.
\end{align}

De fa\c{c}on similaire à la démonstration du Corollaire \ref{cor2}, on peut obtenir une majoration pour $\delta_*(n)$
en utilisant l'inégalité \eqref{eq:thm1_B} du Théorème \ref{thm1}.
\begin{corollaire*}
Pour $x\rightarrow \infty$, on a
\begin{align*}
\sum_{n\leqslant x} \delta_*(n)^{-1} <(1-4,42\times 10^{-4})x .
\end{align*}
\end{corollaire*}

On peut obtenir respectivement une majoration des quatre cas de figure pour les triplets d'entiers cons\'{e}cutifs en utilisant
les minorations \eqref{mino-delta(n)} et \eqref{mino-delta*(n)}, ou en combinant le Théorème \ref{thm1} avec le Corollaire \ref{cor1}.

\begin{corollaire}\label{cor3}
Pour $x\rightarrow \infty$, on a
\begin{align*}
&\big|\big\{n\leqslant x:\, P^+(n-1)>P^+(n)<P^+(n+1)\big\}\big|< \tfrac{2}{3}x,
\\\noalign{\vskip 1mm}
&\big|\big\{n\leqslant x:\, P^+(n-1)<P^+(n)>P^+(n+1)\big\}\big|< \tfrac{2}{3} x,
\\\noalign{\vskip 1mm}
&\big|\big\{n\leqslant x:\, P^+(n-1)<P^+(n)<P^+(n+1)\big\}\big|< (0.8644-8.84\times 10^{-4})x,
\\\noalign{\vskip 1mm}
&\big|\big\{n\leqslant x:\, P^+(n-1)>P^+(n)>P^+(n+1)\big\}\big|< (0.8644-8.84\times 10^{-4})x.
\end{align*}
\end{corollaire}

Dans le paragraphe 2 de cet article, on rappelle la majoration du crible linéaire obtenue par Iwaniec et des résultats
de Hildebrand, Fouvry et Tenenbaum sur les entiers friables ainsi que sur les entiers sans facteur premier dans un intervalle donné.

Le paragraphe 3 porte sur divers théorèmes de type Bombieri-Vinogradov. Nous obtenons notamment pour la suite des entiers
avec un grand facteur premier un niveau de distribution en $x^{4/7-\varepsilon}$ lorsque la moyenne est prise avec un poids bien factorisable. C'est une `` légère ''\\ généralisation du théorème de  Bombieri-Friedlander-Iwaniec qui pourra peut-\^{e}tre servir dans d'autres contextes.

Les paragraphes suivants sont dévolus aux preuves des différents résultats annoncés dans cette introduction.

\begin{Remark}
Peu après la présentation des ces travaux en mai 2017 lors de la
conférence \og Prime Numbers and Automatic Sequences \fg\ au CIRM à Marseille,
Joni Ter\"{a}v\"{a}inen \cite{Ter18} m'a annoncé qu'il avait une autre preuve de
la densité inférieure strictement positive des ensembles étudiés
au Théorème 1. Sa démonstration ne fournit pas de minoration explicite
des densités inférieures mais présente une approche différente
et intéressante sur ce problème.

\end{Remark}

\begin{remerciements}
Ce travail a \'{e}t\'{e} r\'{e}alis\'{e} sous la direction de mes directeurs de th\`{e}se C\'{e}cile Dartyge et Jie Wu. Je les remercie vivement pour les nombreuses suggestions cruciales qu'ils ont propos\'{e}es dans l'\'{e}laboration de ce travail.
\end{remerciements}

\vskip 8mm

\section{Deux lemmes de cribles}

\subsection{Borne supérieure du crible linéaire}\

\vskip 1mm
Dans ce paragraphe nous rappelons un résultat d'Iwaniec sur le crible linéaire. Nous énon\c{c}ons ici seulement la majoration car seule celle-ci sera utilisée dans cet article.

Soient $\mathcal{A}$ une suite finie d'entiers, $\mathcal{P}$ un ensemble de nombres premiers, $z\geqslant 2$ un nombre r\'{e}el, $d$ un entier sans facteur carr\'{e} dont les facteurs premiers appartiennent \`{a} $\mathcal{P}$. Notons
$$
\mathcal{A}_d := \big\{a\in \mathcal{A} \,:\, d\mid a \big\},
\qquad
P_{\mathcal{P}}(z) := \prod_{p<z,\, p\in \mathcal{P}}p.
$$
On souhaite \'{e}valuer
$$
S(\mathcal{A}; \mathcal{P}, z)
:= |\{a\in \mathcal{A} : (a,\, P_{\mathcal{P}}(z))=1\}|.
$$
On suppose que $|\mathcal{A}_d|$ vérifie une formule de la forme
$$
|\mathcal{A}_d|=\frac{w(d)}{d}X + r(\mathcal{A}, d) \quad  \textmd{pour} \; d\mid P_{\mathcal{P}}(z),
$$
o\`{u} $X$ est une approximation de $|\mathcal{A}|$ ind\'{e}pendante de $d$, $w$ une fonction multiplicative v\'{e}rifiant $0<w(p)<p$ pour $p\in \mathcal{P}$, $w(d)d^{-1}X$ un terme principal et
$r(\mathcal{A}, d)$ un terme d'erreur que l'on esp\`{e}re petit en moyenne sur $d$. De plus, on d\'{e}finit
$$
V(z) := \prod_{p<z, \, p\in \mathcal{P}} \bigg(1-\frac{w(p)}{p}\bigg).
$$

On a ainsi \cite{Iwa80b}

\begin{Lemma}\label{lem:sieve}
On suppose qu'il existe une constante $K\geqslant2$ telle que
$$
\prod_{u\leqslant p<v} \bigg(1-\frac{w(p)}{p}\bigg)^{-1}
\le \frac{\log v}{\log u} \bigg(1+\frac{K}{\log u}\bigg)
$$
pour tout $v>u\geqslant 2$.
Alors pour tout $\varepsilon>0$ et $D^{1/2}\geqslant z\geqslant 2$, on a
$$
S(\mathcal{A}; \mathcal{P}, z)
\leqslant XV(z)\bigg\{F\bigg(\frac{\log D}{\log z}\bigg)+ E\bigg\}+
\sum_{\ell<\exp(8/\varepsilon^3)}\, \sum_{d\mid P_{\mathcal{P}}(z)}\lambda_\ell^+(d) r(\mathcal{A}, d),
$$
o\`{u} $F(s) =  2\mathrm{e}^{\gamma}s^{-1} \; (0<s\leqslant 3)$,
$\gamma$ est une constante d'Euler,
$\lambda_\ell^+(d)$ d\'{e}signe un coefficient
bien factorisable de niveau $D$ et d'ordre $1$.
Le terme d'erreur $E$ satisfait
$$
E=O\big(\varepsilon+\varepsilon^{-8} {\rm e}^K(\log D)^{-1/3}\big).
$$
\end{Lemma}
Les $\lambda_\ell^+(d)$ sont les poids de Rosser-Iwaniec. On pourra trouver une définition précise dans \cite{Iwa80b}.
Ici nous indiquons simplement  que $|\lambda_\ell^+(d)|\leqslant 1.$ La notation de fonction bien factorisable est définie
au début du paragraphe 3.2.

\vskip 1mm

\subsection{Entiers sans facteur premier dans un intervalle donné}\

\vskip 1mm

Soit
\begin{equation}\label{def:Pyz}
P(y, z) :=  \prod_{z<p\le y} p.
\end{equation}
On d\'{e}signe pour $z<y\leqslant x$
\begin{align}\label{def:S(x;y,z)}
S(x;\, y, z) := \{n\leqslant x : (n,\, P(y, z))=1\}
\end{align}
l'ensemble des entiers sans facteur premier dans l'intervalle $(z,\, y]$ et n'exc\'{e}dant pas $x$. On note le cardinal
$$
\Psi_0(x;\, y, z):=\big|S(x;\, y,z)\big|.
$$
Alors $\Psi_0(x;\, y, z)$ est \'{e}valu\'{e}e par le lemme suivant (voir \cite[Exercice 299]{Ten08} ou \cite{TenWu14} pour la correction).

\begin{Lemma}\label{lem:S(x;y,z)}
On a
$$
\Psi_0(x;\, y, z) = \vartheta_0(\lambda, u) x \{1+O(1/\log z)\}
$$
uniform\'{e}ment pour $y\geqslant z\geqslant 2$ et $x\geqslant yz$,
o\`{u}
$$
u:=\frac{\log x}{\log y},\qquad \lambda:=\frac{\log z}{\log y}
$$
et
\begin{equation}\label{def:vartheta0}
\vartheta_0(\lambda, u):=\rho(u/\lambda)+\int_{0}^{u}\rho(t/\lambda)\omega(u-t){\rm{d}}t
\end{equation}
avec la convention $\vartheta_0(0, u)=0$.
La fonction de Buchstab $\omega(u)$ est d\'{e}finie comme la solution continue du syst\`{e}me
$$
\begin{cases}
u\omega(u) = 1 & \ {\rm{si}} \ \ 1\leqslant u\leqslant 2,
\\
(u\omega(u))' = \omega(u-1) & \ {\rm{si}} \ \ u>2.
\end{cases}
$$
De plus, nous prolongeons $\omega(u)$ par $0$ pour $u<1$. La fonction de Dickman $\rho(u)$ est d\'{e}finie par l'unique solution continue de l'\'{e}quation diff\'{e}rentielle aux diff\'{e}rences
\begin{equation}\label{def:dickman}
\begin{cases}
\rho(u) = 1 & \ {\rm{si}} \ \ 0\leqslant u\leqslant 1,
\\
u\rho'(u) = -\rho(u-1) & \ {\rm{si}} \ \ u>1.
\end{cases}
\end{equation}
\end{Lemma}

\subsection{Entiers friables}\

\vskip 1mm

Posons
\begin{equation}\label{def:Sxy}
S(x, y) := \{n\leqslant x : P^+(n)\leqslant y\},
\qquad
\Psi(x,  y) := |S(x, y)|
\end{equation}
et
\begin{equation}\label{def:Psixyaq}
\Psi(x, y;\, a,  q):= \sum_{\substack{n\in S(x,\, y)\\ n\equiv a (\text{mod}\, q)}} 1,
\qquad
\Psi_q(x,  y):=\sum_{\substack{n\in S(x,\, y)\\ (n,\, q)=1}} 1.
\end{equation}

Les deux lemmes respectivement  suivants,
dus à Hildebrand \cite[Theorem 1]{Hil86} et à Fouvry-Tenenbaum \cite[Théorème 1]{FouTen91},
serviront dans la démonstration du Théorème \ref{thm1}.

\begin{Lemma}\label{lem2.3}
Soit $\varepsilon>0$.
Alors on a
$$
\Psi(x, y) = x \rho(u) \bigg\{1+O_{\varepsilon}\bigg(\frac{\log(u+1)}{\log y}\bigg)\bigg\}
$$
uniformément pour
$$
x\ge x_0(\varepsilon),
\qquad
\exp\{(\log_2x)^{5/3+\varepsilon}\}\le y\le x,
\leqno(H_{\varepsilon})
$$
o\`{u} $u=\log x/ \log y$ et $\rho(u)$ est d\'{e}finie par \eqref{def:dickman}.
\end{Lemma}

\vskip 1mm

\begin{Lemma}\label{lem2.4}
Soit $\varepsilon>0$.
Alors on a
$$
\Psi_q(x, y) = \frac{\varphi(q)}{q} \Psi(x, y) \bigg\{1+O\bigg(\frac{\log_2(qy)\log_2x}{\log y}\bigg)\bigg\}
$$
uniformément pour
$$
x\ge x_0(\varepsilon),
\qquad
\exp\{(\log_2x)^{5/3+\varepsilon}\}\le y\le x
\leqno(H_{\varepsilon})
$$
et
$$
\log_2(q+2)
\le \bigg(\frac{\log y}{\log(u+1)}\bigg)^{1-\varepsilon}.
\leqno(Q_{\varepsilon})
$$
\end{Lemma}

\vskip 8mm

\section{Deux théorèmes de type Bombieri-Vinogradov}

Dans cette section, nous démontrons deux théorèmes de type Bombieri-Vinogradov
que l'on utilisera dans la démonstration du Théorème \ref{thm2}.

\subsection{Théorème de type Bombieri-Vinogradov pour $S(x;\,y,z)$}\

\vskip 1mm

Soit $S(x;\, y, z)$ l'ensemble défini comme dans \eqref{def:S(x;y,z)}.
Notre théorème de type Bombieri-Vinogradov pour $S(x;\,y,z)$ est le suivant.

\begin{proposition}\label{prop:S(x;y,z)}
Pour tout $A>0$ et tout $\varepsilon>0$, il existe une constante $B=B(A)>0$ telle que l'on ait
$$
\sum_{q\leqslant x^{1/2}/(\log x)^B} \max_{t\leqslant x} \max_{(a,\, q)=1}
\bigg|\sum_{\substack{n\in S(t;\, y, z)\\n\equiv a ({\rm mod}\, q) }} 1
- \frac{1}{\varphi(q)}\sum_{\substack{n\in S(t;\, y, z)\\(n,\,q)=1}}1\bigg|
\ll_{A, \varepsilon} \frac{x}{(\log x)^A}
$$
uniformement pour
\begin{equation}\label{dom:xyz}
2\leqslant z \leqslant y \leqslant x
\qquad\text{et}\qquad
\exp\{(\log x)^{2/5+\varepsilon}\}\leqslant y\leqslant x,
\end{equation}
où $\varphi(q)$ est la fonction d'Euler.
\end{proposition}

Pour démontrer cette proposition, rappelons d'abord un résultat général de Motoshashi \cite{Mot76}.
Soit $f$ une fonction arithm\'{e}tique vérifiant les propri\'{e}t\'{e}s suivantes:
\begin{itemize}
\item[$(\mathscr{A})$]
$f(n)\ll \tau(n)^C$, o\`{u} $\tau(n)$ d\'{e}signe la fonction diviseur et $C$ est une constante.
\item[$(\mathscr{B})$]
Si le conducteur d'un caract\`{e}re de Dirichlet non principal $\chi$  est
$O((\log x)^D)$, alors
$$
\sum_{n\leqslant x}f(n)\chi(n)\ll x(\log x)^{-3D}
\qquad
(x\ge 2),
$$
o\`{u} $D$ est une constante positive.
\item[$(\mathscr{C})$]
Soit
$$
E_f(y;\, q, a) := \sum_{\substack{n\leqslant y\\ n\equiv a({\rm mod}\, q)}}f(n)
-\frac{1}{\varphi(q)}\sum_{\substack{n\leqslant y\\ (n, q)=1 }}f(n),
$$
alors pour tout $A>0$, il existe une constante $B=B(A)>0$ telle que
$$
\sum_{q\leqslant x^{1/2}/(\log x)^B} \max_{y\leqslant x} \max_{(a,\, q)=1}
\big|E_f(y;\, q, a)\big|
\ll \frac{x}{(\log x)^A}
\qquad
(x\ge 2).
$$
Les constantes $A, B, C, D$ ne d\'{e}pendent que la fonction $f$.
\end{itemize}

\vskip 1mm

Le résultat suivant est d\^{u} à Motohashi \cite[Theorem 1]{Mot76}.

\begin{Lemma}\label{lem:Motohashi}
Soient $f$ et $g$ deux fonctions arithm\'{e}tiques vérifiant les propri\'{e}t\'{e}s $(\mathscr{A})$, $(\mathscr{B})$ et $(\mathscr{C})$. Alors la convolution
multiplicative $f\ast g$ vérifie aussi $(\mathscr{A})$, $(\mathscr{B})$ et $(\mathscr{C})$.
\end{Lemma}

De manière analogue à $P^+(n)$, désignons par $P^-(n)$ le plus petit facteur premier d'un entier $n\geqslant1$
avec la convention $P^-(1)=\infty$.
Posons
\begin{equation}\label{def:tildeSxy}
\widetilde{S}(x, y)
:= \{n\leqslant x : P^-(n)> y\},
\qquad
\Phi(x, y) := |\widetilde{S}(x, y)|
\end{equation}
et
\begin{equation}\label{def:Phixyaq}
\Phi(x, y;\, a,  q)
:= \sum_{\substack{n\in \widetilde{S}(x,\, y)\\ n\equiv a (\text{mod}\, q)}}1,
\qquad
\Phi_q(x,  y)
:= \sum_{\substack{n\in \widetilde{S}(x,\, y)\\ (n,\, q)=1}}1.
\end{equation}

Nous utilisons des théorèmes de type Bombieri-Vinogradov pour les entiers criblés et pour les entiers friables.

\begin{Lemma}\label{lem:BV-S(x,y)}
Pour tout $A>0$, il existe une constante $B=B(A)>0$ telle que l'on ait
\begin{align}
\sum_{q\leqslant x^{1/2}/(\log x)^B}
\max_{z\leqslant x} \max_{(a,\, q)=1}
\bigg|\Psi(z, y;\, a,  q)-\frac{\Psi_q(z,  y)}{\varphi(q)}\bigg|
\ll_A \frac{x}{(\log x)^A}
\label{eq:BV-S(x,y)}
\\
\sum_{q\leqslant x^{1/2}/(\log x)^B} \max_{z\leqslant x} \max_{(a,\, q)=1}
\bigg|\Phi(z, y;\, a, q)-\frac{\Phi_q(z,  y)}{\varphi(q)}\bigg|
\ll_A \frac{x}{(\log x)^A}
\label{eq:BV-tildeS(x,y)}
\end{align}
uniformément pour $x\geqslant y\geqslant 2$.
\end{Lemma}

La formule \eqref{eq:BV-tildeS(x,y)} a été démontrée par Wolke \cite{Wol73}, qui dans le m\^{e}me article annonce une formule \'{e}quivalente pour les friables.
La formule \eqref{eq:BV-S(x,y)} a été démontrée par Fouvry-Tenenbaum \cite{FouTen91}.
On trouvra dans l'article \cite{FouTen96} de Fouvry-Tenenbaum (voir \'{e}galement les travaux r\'{e}cents de Drappeau \cite{Dra15}) une formule avec une majoration en $\Psi(x, y)(\log x)^{-A}$ à la place de $x(\log x)^{-A}$ lorsque $y>\exp\{(\log x)^{2/3+\varepsilon}\}$. Cependant, dans nos preuves, l'inégalité \eqref{eq:BV-S(x,y)} sera suffisante. En particulier, nous exploitons l'uniformit\'{e} en $``\max\limits_{z\leqslant x}"$ dans \eqref{eq:BV-S(x,y)} et qui n'appara\^{i}t pas dans \cite{FouTen96}.

\vskip 1mm
Nous somme maintenant prêts pour la preuve de la Proposition \ref{prop:S(x;y,z)}.

\begin{proof}
Soient $\lambda$ la fonction caract\'{e}ristique de l'ensemble $S(x;\, y, z)$, c'est-à-dire,
$$
\lambda(n)
:= \begin{cases}
1 & \quad  \textmd{si} \ \ n\in S(x;\, y, z),
\\
0 &  \quad  \textmd{sinon}.
\end{cases}
$$
Définissons deux fonctions arithm\'{e}tiques $v_z$ et $u_y$ par
$$
v_z(n):=\left\{
\begin{array}{ll}
    1 & \quad  \textmd{si} \ \ P^+(n)\leqslant z, \\
    0 &  \quad  \textmd{sinon},
  \end{array}
\right.
\qquad
u_y(n):=\left\{
\begin{array}{ll}
    1 & \quad  \textmd{si} \ \ P^-(n)> y, \\
    0 &  \quad  \textmd{sinon}.
  \end{array}
\right.
$$
Alors $v_z$ et $u_y$ sont multiplicatives et on a
\begin{equation}\label{Proof-Prop1:A}
\lambda=v_z\ast u_y.
\end{equation}

En fait, si $n\not\in S(x;\, y, z)$, il existe un premier $p$ tel que $p\mid n$ et $z\leqslant p <y$.
D'où
$$
v_z\ast u_y(n)
= \sum_{\substack{d_1d_2=n\\ p|d_1\ \text{ou}\ p|d_2}} v_z(d_1)u_y(d_2)
= 0 = \lambda(n).
$$
Si $n\in S(x;\, y, z)$, alors cet entier peut \^etre écrit de manière unique sous la forme
$$
n=n_1n_2,
\qquad
P^+(n_1)\leqslant z,
\qquad
P^-(n_2)> y.
$$
Ainsi
$$
v_z\ast u_y(n)
= \sum_{d_1d_2=n} v_z(d_1)u_y(d_2)
=  v_z(n_1)u_y(n_2)
= 1
= \lambda(n).
$$
Donc pour démontrer le résultat souhaité, gr\^{a}ce au Lemme \ref{lem:Motohashi},
il suffit de vérifier que les deux fonctions $v_z$ et $u_y$ possèdent les propriétés $(\mathscr{A})$, $(\mathscr{B})$ et $(\mathscr{C})$.

La première est triviale.

Soient $D>0$, $1<q\leqslant (\log x)^D$ et $c_0, c_1, c_2$ des constantes strictement  positives, alors pour tout caract\`{e}re de Dirichlet $\chi$ non
principal modulo $q$, on a, d'après le Th\'{e}or\`{e}me 4 de \cite{FouTen91}
\begin{align*}
\sum_{\substack{n\leqslant x\\ P^+(n)\leqslant z}}\chi(n)
\ll \Psi(x,\, z) \text{e}^{-c_1\sqrt{z}}
\ll x(\log x)^{-3D}
\end{align*}
sous la condition
$$
x\ge 3,
\qquad
\exp\{c_0(\log_2x)^2\}\le z\le x.
$$
De plus, si $2\le z<\exp\{c_0(\log_2x)^2\}$,
 le Théorème III.5.1 de \cite{Ten08} implique trivialement que
\begin{align*}
\sum_{\substack{n\leqslant x\\ P^+(n)\leqslant z}}\chi(n)
& \ll \Psi(x,\, z)
\ll x \text{e}^{-(\log x)/(2\log z)}
\ll x(\log x)^{-3D}.
\end{align*}
Cela montre que $v_z$ possède la propri\'{e}t\'{e} $(\mathscr{B})$ pour $2\leqslant z\leqslant x$.
De mani\`{e}re similaire, par le r\'{e}sultat de \cite[Theorem 1]{Xuan00}, $u_y$ satisfait $(\mathscr{B})$ pour
$\exp\{(\log x)^{2/5+\varepsilon}\}\leqslant y\leqslant x^{1/2}.$
Si $y>x^{1/2}$, $m$ est premier et on sait que $u_y(m)$ satisfait $(\mathscr{B})$ (voir \cite{Ten08}).
Donc $u_y$ satisfait $(\mathscr{B})$ pour
$$
\exp\{(\log x)^{2/5+\varepsilon}\}\leqslant y\leqslant x.
$$

On en d\'{e}duit ensuite que $v_z$ et $u_y$ satisfont $(\mathscr{C})$ en utilisant \eqref{eq:BV-S(x,y)} et \eqref{eq:BV-tildeS(x,y)} du Lemme \ref{lem:BV-S(x,y)}  respectivement. Finalement, \`{a} l'aide du Lemme \ref{lem:Motohashi} on d\'{e}duit que
$\lambda=v_z\ast u_y$ satisfait des propri\'{e}t\'{e}s $(\mathscr{A})$, $(\mathscr{B})$ et $(\mathscr{C})$ sous la
condition
$$
2\leqslant z\leqslant y\leqslant x,\qquad \exp\{(\log x)^{2/5+\varepsilon}\}\leqslant y\leqslant x.
$$
La Proposition \ref{prop:S(x;y,z)} est ainsi d\'{e}montr\'{e}e.
\end{proof}

\subsection{Théorème de type Bombieri-Vinogradov avec une fonction bien factorisable}\

\vskip 1mm

Pour un entier positif $k$, on d\'{e}finit $\tau_k(n)$ par
$$
\tau_k(n)=\sum_{n=n_1n_2\cdots n_k}1.
$$
Une fonction arithm\'{e}tique $\lambda(q)$ est dite de niveau $Q$ et d'ordre $k$ si
$$
\lambda(q)=0\quad(q>Q)
\qquad\text{et}\qquad
|\lambda(q)|\leqslant \tau_k(q)\quad(q\leqslant Q).
$$
On dit que $\lambda$ est bien factorisable de niveau $Q$
si pour toute d\'{e}composition $Q=Q_1Q_2$ avec $Q_1, Q_2 \geqslant 1$,
il existe deux fonctions arithm\'{e}tiques $\lambda_1, \lambda_2$ de niveaux $Q_1, Q_2$ et d'ordre $k$ telles que
$$
\lambda=\lambda_1\ast \lambda_2.
$$

Pour $q\in \mathbb{N}^*$ et $(a, q)=1$, on définit
\begin{equation}\label{pi(y;l,a,q)}
\pi(x;\, \ell, a, q) := \sum_{\substack{\ell p\leqslant x\\ \ell p\equiv a(\text{mod}\, q)}}1,
\end{equation}
pour $\ell=1$, on retrouve $\pi(x;\, a, q)$ la fonction de compte des nombres premiers dans les progressions arithmétiques.

Le théorème de Bombieri-Vinogradov assure  que $\pi(x;\, a, q)$ est proche de $\text{li}(x)/\varphi(q)$ en moyenne pour
$q\leqslant x^{1/2}/(\log x)^A$. Bombieri, Friedlander et Iwaniec montrent que la borne $q\leqslant x^{1/2}/(\log x)^A$ peut \^{e}tre
remplacée par $x^{4/7-\varepsilon},$ si on insère un poids bien factorisable. Nous présentons ici une légère généralisation de résultat de \cite{BomFriIwa86}.

\begin{proposition}\label{prop:4/7}
Soient $a\in \mathbb{Z}^*$, $A>0$ et $\varepsilon>0$.
Alors pour toute fonction bien factorisable $\lambda(q)$ de niveau $Q$, la majoration suivante
\begin{align*}
\sum_{(a,\, q)=1}\lambda(q) \sum_{\substack{L_1\leqslant \ell\leqslant L_2\\ (\ell,\, q)=1}}
\bigg(\pi(x;\, \ell, a, q) -\frac{{\rm{li}}(x/\ell)}{\varphi(q)}\bigg)
\ll_{a, A, \varepsilon} \frac{x}{(\log x)^A}
\end{align*}
ait lieu pour
$$
Q=x^{4/7-\varepsilon},\qquad 1\leqslant L_1 \leqslant L_2\leqslant x^{1-\varepsilon}.
$$
La constante implicite d\'{e}pend au plus de $a$, $A$ et $\varepsilon$.
\end{proposition}

\begin{proof}
En approchant $\text{li}(x/\ell)$ par $\sum_{\ell n\leqslant x}\Lambda(n)$ puis en observant que la contribution des $n$ tels que $(n,\, qP(z))\neq 1$ est négligeable, on vérifie qu'il suffit de montrer
\begin{align}\label{Theorem3/5aim}
\sum_{(a,\, q)=1}\lambda(q) \sum_{\substack{L_1\leqslant \ell\leqslant L_2\\ (\ell,\, q)=1}}
\bigg(\sum_{\substack{\ell n\leqslant x\\\ell n\equiv a({\rm mod}\,q)\\ (n,\, P(z))=1 }} \Lambda(n)
-\frac{1}{\varphi(q)}\sum_{\substack{\ell n\leqslant x\\ (n,\, qP(z))=1}}\Lambda(n) \bigg)
\ll_{a, A, \varepsilon} \frac{x}{(\log x)^A}
\end{align}
avec $z=\exp(\log x/\log_2x)$.

Puisque la d\'{e}monstration est tr\`{e}s proche de celle de \cite[Theorem 10]{BomFriIwa86},
nous indiquons les points essentiels et les différences entre les deux démonstrations.
Pour cela, on introduit
\begin{align*}
& \Delta(L\mid M_1, \ldots, M_j\mid N_1, \ldots, N_j;\, q,\, a)
\\\noalign{\vskip 0,5mm}
& \hskip 5mm
:= \mathop{{\sum}^*}_{\substack{\ell m_1\ldots m_jn_1\ldots n_j\equiv a ({\rm mod}\,q)\\ \ell\in \mathscr{L},\, m_i\in \mathscr{M}_i ,\, n_i\in \mathscr{N}_i}}
\mu(m_1)\ldots \mu(m_j)
- \frac{1}{\varphi(q)} \mathop{{\sum}^*}_{\substack{(\ell m_1\ldots m_jn_1\ldots n_j,\, q)=1\\ \ell\in \mathscr{L},\, m_i\in \mathscr{M}_i ,\, n_i\in \mathscr{N}_i}}
\mu(m_1)\ldots \mu(m_j),
\end{align*}
où $\sum^*$ d\'{e}signe que la somme est restreinte \`{a} des entiers $m_1, \ldots, m_j, n_1, \ldots, n_j$ sans facteur premier $<z$, et $\mathscr{L}, \mathscr{M}_i, \mathscr{N}_i$ sont les intervalles suivants
$$
\mathscr{L} := [(1-\Delta)L,\, L[,
\quad
\mathscr{M}_i := [(1-\Delta)M_i,\, M_i[,
\quad
N_i = [(1-\Delta)N_i,\, N_i[
$$
avec
$$
\ell m_1\ldots m_jn_1\ldots n_j=x,\qquad
\max(M_1, \ldots, M_j)< x^{1/7}
$$
et $\Delta=(\log x)^{-A_1}$
($A_1$ est une constante assez grande d\'{e}pendant de $A$).

Dans la démonstration de \cite[Theorem 10]{BomFriIwa86},
on remplace, pour $j\leqslant J=7$,
$$
\mathscr{E}(M_1, \ldots, M_j\mid N_1, \ldots, N_j)
:= \mathop{\sum_{q_1\sim Q}\,\sum_{q_2\sim R}}_{(a, \, q_1q_2)=1}
\gamma_{q_1}\delta_{q_2}\, \Delta(M_1, \ldots, M_j\mid N_1, \ldots, N_j;\, q_1q_2,\, a)
$$
par
$$
\mathscr{E}(L\mid M_1, \ldots, M_j\mid N_1, \ldots, N_j)
:= \mathop{\sum_{q_1\sim Q}\,\sum_{q_2\sim R}}_{(a, \, q_1q_2)=1}
\gamma_{q_1}\delta_{q_2}\, \Delta(L\mid M_1, \ldots, M_j\mid N_1, \ldots, N_j;\, q_1q_2,\, a).
$$

On va montrer que
\begin{align}\label{Thmaim}
\mathscr{E}(L\mid M_1, \ldots, M_j\mid N_1, \ldots, N_j)\ll x(\log x)^{-A_2}
\end{align}
pour tout $A_2$. Notons
$$ L=x^{\nu_0},\quad M_i=x^{\mu_i},\quad N_i=x^{\nu_i}$$
avec
\begin{equation}\label{munu}
\begin{aligned}
& 0\leqslant \mu_j \leqslant \cdots \leqslant \mu_1 \leqslant \tfrac{1}{7},
\qquad
0\leqslant \nu_j \leqslant \cdots \leqslant \nu_1,
\qquad
0\leqslant \nu_0<1-\varepsilon,
\\\noalign{\vskip 1mm}
& \hskip 23mm
\mu_1+ \cdots + \mu_j + \nu_1+\cdots +\nu_j+\nu_0=1.
\end{aligned}
\end{equation}

On utilise ensuite un argument combinatoire similaire \`{a} \cite[Theorem 10]{BomFriIwa86}.
La diff\'{e}rence est qu'on ne peut pas appliquer \cite[Theorem 1]{BomFriIwa86} et
\cite[Theorem 2]{BomFriIwa86} \`{a} $L=x^{\nu_0}$,
car on n'a pas la condition $(\ell,\, P(z))=1$ pour $\ell$.
Pour surmonter cette difficulté,
nous observons tout d'abord que, si $\nu_0\geqslant \tfrac{3}{7}$,
on peut appliquer le théorème 5 de \cite{BomFriIwa86} avec
\begin{align*}
M=L=x^{\nu_0} \geqslant x^{3/7},
\qquad
\max(Q, R)\leqslant x^{2/7-2\varepsilon}
\end{align*}
pour obtenir \eqref{Thmaim}.
Sinon, c'est-à-dire, dans le cas où $\nu_0<\tfrac{3}{7}$,
on utilise le théorème 1 ou 2 de \cite{BomFriIwa86} avec
$$
R=x^{-\varepsilon}N,\qquad Q\leqslant x^{4/7-4\varepsilon}N^{-1}
$$
selon
\begin{align}\label{N11}
x^{2/7-\varepsilon} < N < x^{3/7+\varepsilon}
\end{align}
ou
\begin{align}\label{N12}
x^{1/7-\varepsilon} < N < x^{2/7+\varepsilon}.
\end{align}
Si \eqref{munu} a une somme partielle $\lambda$ de $\mu_1, \ldots, \mu_j, \nu_1, \ldots, \nu_j$ avec
\begin{align}\label{lambda1}
\tfrac{1}{7}\leqslant \lambda \leqslant \tfrac{3}{7},
\end{align}
on compl\`{e}te la preuve d'apr\`{e}s \eqref{N11} ou \eqref{N12}.
Sinon, on peut supposer qu'il n'y a pas de somme partielle de $\mu_1, \ldots, \mu_j, \nu_1, \ldots, \nu_j$ de \eqref{munu} dans \eqref{N11} ou \eqref{N12}.
Ainsi, toutes les $\mu_i$ et $\nu_i$ avec $\nu_i\leqslant \tfrac{1}{7}$ donnent un produit $\eta$ avec
$$
\eta<\tfrac{1}{7},
$$
d'o\`{u} l'on peut d\'{e}duire que
$$
\eta+\nu_0<\tfrac{4}{7}\cdot
$$
On a ainsi $\nu_1\geqslant \tfrac{3}{7}$. En utilisant \cite[Theorem $5^*$]{BomFriIwa86} avec
\begin{align*}
M=N_1=x^{\nu_1} \geqslant x^{3/7},\qquad
\max(Q, R)\leqslant x^{2/7-2\varepsilon},
\end{align*}
ce qui termine la d\'{e}monstration de la Proposition 2.
\end{proof}

\vskip 1mm

Le lemme suivant, d\^{u} à Pan-Ding-Wang \cite{PanDingWang75}, sera aussi utile dans la démonstration du
Théorème \ref{thm2}(ii).

\begin{Lemma}\label{lem:PanDingWang}
Soient $\alpha\in\,]0, 1[$ et $f(\ell)$ une fonction arithm\'{e}tique v\'{e}rifiant $|f(\ell)|\leqslant 1$.
Alors pour tout $A>0$, il existe une constante $B=B(A)>0$ telle que l'on ait
$$
\sum_{q\leqslant Q}\, \max_{(a,\, q)=1}\, \max_{y\leqslant x}\,
\bigg|\sum_{\substack{L_1<\ell\leqslant L_2\\(\ell,\, q)=1}}f(\ell)
\bigg(\pi(y;\, \ell, a, q)-\frac{{\rm{li}}(y/\ell)}{\varphi(q)}\bigg)\bigg|
\ll_{\alpha, A} \frac{x}{(\log x)^A}
$$
uniformément pour
$$
Q=x^{1/2}/(\log x)^B,\qquad (\log x)^{2B}\le L_1\leqslant L_2\le x^{\alpha},
$$
où la constante implicite ne d\'{e}pend que de $\alpha$ et $A$.
\end{Lemma}

\vskip 8mm

\section{D\'{e}monstration du Th\'{e}or\`{e}me \ref{thm1}: cas des $n$ tels que $P^+(n-1)>P^+(n)<P^+(n+1)$}

Soit $y=x^{\alpha}$ où $0<\alpha <\frac{1}{4}$ est un paramètre à choisir plus tard.
Il est clair que
\begin{equation}\label{3.1}
\sum_{\substack{n\leqslant x\\ P^+(n-1)>P^+(n)<P^+(n+1)}} 1
\geqslant \sum_{\substack{n\in S(x,\, y)\\ (n\pm 1,\ P(x, y))>1}} 1,
\end{equation}
où $S(x,\, y)$ et $P(x, y)$ sont définis respectivement par \eqref{def:Sxy} et \eqref{def:Pyz}.

Pour $n\leqslant x$ et $z<y$, on a
\begin{align}\label{def:omega}
\omega(n;\, y, z)
:= \sum_{\substack{z<p\leqslant y\\ p\mid n}} 1
\leqslant \frac{\log x}{\log z},
\end{align}
d'où
\begin{align}\label{omega-pro}
\bigg(\frac{\log x}{\log z}\bigg)^{-1} \omega(n;\, y, z)
\left\{
\begin{array}{ll}
    \leqslant 1 & \quad  \textmd{si} \ \  (n,\ P(z, y))>1,
\\\noalign{\vskip 2mm}
    =0 &  \quad  \textmd{sinon} .
  \end{array}
\right.
\end{align}
Ainsi on peut d\'{e}tecter les conditions $(n\pm 1,\, P(x, y))>1$ dans la formule \eqref{3.1}
par \eqref{omega-pro} :
\begin{equation}\label{criblepondere}
\begin{aligned}
\sum_{\substack{n\leqslant x\\ P^+(n-1)>P^+(n)<P^+(n+1)}} 1
& \geqslant \sum_{n\in S(x,\, y)} \frac{\omega(n-1;\, x, y)}{(\frac{\log x}{\log y})}
\cdot \frac{\omega(n+1;\, x, y)}{(\frac{\log x}{\log y})}
\\\noalign{\vskip 0mm}
& \geqslant  {\alpha}^2
\mathop{\sum_{y<p_1\leqslant x} \
\sum_{y<p_2\leqslant x}}_{\substack{p_1p_2\leqslant x^{1/2}/(\log x)^B\\ p_1\not=p_2}}
\sum_{\substack{n\in S(x,\, y)\\ n\equiv 1\,(\text{mod}\, p_1) \\ n\equiv -1\,(\text{mod}\, p_2)}}1
\end{aligned}
\end{equation}
o\`{u} $B$ est une constante convenable.  Pour la somme int\'{e}rieure du membre de droite,
on utilise le th\'{e}or\`{e}me des restes chinois. Il existe un entier $a$ tel que les congruences
$n\equiv 1\,(\text{mod}\, p_1),\, n\equiv -1\,(\text{mod}\, p_2)$ soient équivalents à
$n\equiv a\,(\textrm{mod}\, p_1p_2)$.
On en d\'{e}duit donc
\begin{equation}\label{thm1:S1S2}
\begin{aligned}
\sum_{\substack{n\leqslant x\\ P^+(n-1)>P^+(n)<P^+(n+1)}}1
& \geqslant {\alpha}^2
\mathop{\sum_{y<p_1\leqslant x} \
\sum_{y<p_2\leqslant x}}_{\substack{p_1p_2\leqslant x^{1/2}/(\log x)^B\\ p_1\not=p_2}}
\sum_{\substack{n\in S(x,\, y)\\ n\equiv a(\text{mod}\, p_1p_2) }}1
\\\noalign{\vskip 1mm}
& = {\alpha}^2 (\mathcal{S}_1+\mathcal{S}_2),
\end{aligned}
\end{equation}
o\`{u} (avec la notation \eqref{def:Psixyaq})
\begin{align*}
\mathcal{S}_1
& := \mathop{\sum_{y<p_1\leqslant x} \
\sum_{y<p_2\leqslant x}}_{\substack{p_1p_2\leqslant x^{1/2}/(\log x)^B\\ p_1\not=p_2}}
\frac{\Psi_{p_1p_2}(x,\, y)}{\varphi(p_1p_2)},
\\
\mathcal{S}_2
& := \mathop{\sum_{y<p_1\leqslant x} \
\sum_{y<p_2\leqslant x}}_{\substack{p_1p_2\leqslant x^{1/2}/(\log x)^B\\ p_1\not=p_2}}
\bigg(\Psi(x, y;\, a, p_1p_2)- \frac{\Psi_{p_1p_2}(x, y)}{\varphi(p_1p_2)}\bigg).
\end{align*}

Pour $\mathcal{S}_2$, en utilisant l'in\'{e}galit\'{e} de Cauchy-Schwarz on obtient
\begin{align*}
\mathcal{S}_2 \le \sum_{q\leqslant x^{1/2}/(\log x)^B} \tau(q)
\bigg|\Psi(x, y;\, a,  q)-\frac{\Psi_q(x, y)}{\varphi(q)}\bigg|
\leqslant (\mathcal{S}_{21}\mathcal{S}_{22})^{1/2},
\end{align*}
o\`{u}
\begin{align*}
\mathcal{S}_{21}
& := \sum_{q\leqslant x^{1/2}/(\log x)^B}
\bigg|\Psi(x, y;\, a,  q)-\frac{\Psi_q(x, y)}{\varphi(q)}\bigg|,
\\\noalign{\vskip 1mm}
\mathcal{S}_{22}
& := \sum_{\substack{q\leqslant x^{1/2}/(\log x)^B\\\Omega(q)\leqslant2 }} \tau(q)^2
\bigg|\Psi(x, y;\, a,  q)-\frac{\Psi_q(x, y)}{\varphi(q)}\bigg|.
\end{align*}
$\Omega(n)$ est la fonction dénombrant le nombre total des facteurs premiers de $n$ compté
avec leur ordre de multiplicité.

Pour $\mathcal{S}_{22}$ on utilise une majoration triviale :
$$
\Psi(x, y;\, a,  q) + \frac{\Psi_q(x, y)}{\varphi(q)}
\ll \frac{x}{q} + 1,
$$
et pour $\mathcal{S}_{21}$ on applique le Lemme \ref{lem:BV-S(x,y)}.
Cela donne
\begin{equation}\label{thm1:S2}
\begin{aligned}
\mathcal{S}_2
& \ll \bigg(\frac{x}{(\log x)^A}\bigg)^{1/2}
\bigg\{\sum_{q\leqslant x^{1/2}} \tau(q)^2 \,
\bigg(\frac{x}{q} + 1\bigg)\bigg\}^{1/2}
\\\noalign{\vskip 1mm}
& \ll \frac{x}{(\log x)^{A/2-3}}\cdot
\end{aligned}
\end{equation}

Dans $\mathcal{S}_1$ on peut retirer la condition $p_1\neq p_2$ au prix d'une erreur inférieur à:
$$
\sum_{y<p\leqslant x^{1/4}} \frac{\Psi_{p^2}(x,\, y)}{\varphi(p^2)}
\ll \sum_{y<p\leqslant x^{1/4}} \frac{\Psi(x,\, y)}{p^2}
\ll \frac{x^{1-\alpha}}{\log x}.
$$
On \'{e}value ensuite $\mathcal{S}_1$ à l'aide du Lemme \ref{lem2.4} sur les entiers friables.
\begin{align*}
\mathcal{S}_1
& = \mathop{\sum_{y<p_1\leqslant x} \
\sum_{y<p_2\leqslant x}}_{\substack{p_1p_2\leqslant x^{1/2}/(\log x)^B}}
\frac{\Psi_{p_1p_2}(x,\, y)}{\varphi(p_1p_2)} + o(x)
\\\noalign{\vskip 1mm}
& = x \rho\bigg(\frac{\log x}{\log y}\bigg)
\sum_{y<p_1\leqslant x^{1/2}/y(\log x)^B} \frac{1}{p_1}
\sum_{y<p_2\leqslant x^{1/2}/p_1(\log x)^B} \frac{1}{p_2} +o(x)
\\\noalign{\vskip 1mm}
& = x \rho\bigg(\frac{1}{\alpha}\bigg) \sum_{y<p_1\leqslant x^{1/2}/y(\log x)^B}
\frac{1}{p_1} \log\bigg(\frac{\frac{1}{2}\log x- \log p_1}{\log y}\bigg) +o(x),
\end{align*}
o\`{u} on a appliqu\'{e} la formule suivante
$$
\sum_{p\leqslant x}\frac{1}{p}=\log_2x+c_1+O\bigg(\frac{1}{\log x}\bigg)
$$
avec $c_1\approx 0,261 497$ la constante de Meissel-Mertens.
Par une int\'{e}gration par parties, on a
\begin{equation}\label{thm1:S1}
\begin{aligned}
\mathcal{S}_1
& = x \rho\bigg(\frac{1}{\alpha}\bigg) \int_{y}^{x^{1/2}/(y(\log x)^B)}
\log\bigg(\frac{\frac{1}{2}\log x- \log t}{\log y}\bigg)\, \text{d}\sum_{p\leqslant t}\frac{1}{p} +o(x)
\\\noalign{\vskip 1mm}
& = x \rho\bigg(\frac{1}{\alpha}\bigg) \int_{y}^{x^{1/2}/(y(\log x)^B)}
\log\bigg(\frac{\frac{1}{2}- \frac{\log t}{\log x}}{\alpha}\bigg)
\frac{\text{d} t}{t\log t} + o(x)
\\\noalign{\vskip 1mm}
& = x \rho\bigg(\frac{1}{\alpha}\bigg) \alpha
\int_1^{1/(2\alpha)-1} \frac{\log t}{\frac{1}{2}-\alpha t} \, \text{d} t + o(x).
\end{aligned}
\end{equation}

En reportant \eqref{thm1:S1} et \eqref{thm1:S2} dans \eqref{thm1:S1S2}, on d\'{e}duit que
\begin{align}\label{thm3:A}
\sum_{\substack{n\leqslant x\\ P^+(n-1)>P^+(n)<P^+(n+1)}}1
& \geqslant \bigg\{\rho\bigg(\frac{1}{\alpha}\bigg) \alpha^3
\int_1^{1/(2\alpha)-1} \frac{\log t}{\frac{1}{2}-\alpha t} \, \text{d} t  + o(1)\bigg\} x.
\end{align}
\`{A} l'aide de \emph{Mathematica 9.0}, on peut trouver la constante $1,063\times 10^{-7}$
avec $\alpha\approx \frac{1}{4,6}$.
La preuve du Th\'{e}or\`{e}me \ref{thm1} est compl\`{e}te.

\vskip 8mm

\section{D\'{e}monstration du Th\'{e}or\`{e}me \ref{thm1}: cas des $n$ tels que $P^+(n-1)<P^+(n)>P^+(n+1)$}

Soient $\alpha, \beta, \gamma$ tels que
\begin{align}\label{cond:beta1gamma1gamma2}
\frac{1}{2}<\alpha\leqslant 1,
\qquad
1-\alpha\leqslant \beta<\gamma<\frac{1}{3}
\end{align}
trois param\`{e}tres \`{a} choisir plus tard. \'{E}tant donn\'{e}s un entier $m$ et deux nombres premiers distincts
$p_1$ et $p_2$ v\'{e}rifiant
\begin{align}\label{cond:mp1p2}
1\leqslant m \leqslant x^{1-\alpha},
\qquad
x^{\beta} < p_1,\, p_2 \leqslant x^{\gamma},
\end{align}
on consid\`{e}re le syst\`{e}me d'\'{e}quations de congruences:
\begin{align}\label{p1p2}
mp-1\equiv 0\,(\textrm{mod}\, p_1),
\qquad
mp+1\equiv 0\,(\textrm{mod}\, p_2).
\end{align}
En remarque que les conditions \eqref{cond:beta1gamma1gamma2} et \eqref{cond:mp1p2} garantissent que
$(p_1p_2, m)=1$.
D'apr\`{e}s le th\'{e}or\`{e}me chinois, il existe $b<p_1p_2$ (dépendant de $m$) tel que
\begin{align}\label{b0p1p2}
p\equiv b\, (\textrm{mod}\, p_1p_2).
\end{align}
Pour ces $m$ et $p>x^{\alpha}$, on peut donc d\'{e}duire que
$$
P^+(mp-1)\leqslant \max\{p_1,\, x/p_1\} \leqslant \max\big\{x^{1/3},\, x^{1-\beta}\big\}
\leqslant x^{\alpha} < p.
$$
De même,
$$
P^+(mp+1)\leqslant \max\{p_2,\, x/p_2\}  < p.
$$
Ce qui implique
$$
P^+(mp-1)<P^+(mp)>P^+(mp+1).
$$
Ainsi, en reprenant les poids $\omega(n;\, y, z)$ définis par \eqref{def:omega}, on a d'après \eqref{omega-pro},
\begin{align*}
\sum_{\substack{n\leqslant x\\ P^+(n-1)<P^+(n)>P^+(n+1)\\ x^{\alpha}<P^+(n)\leqslant x}} 1
& \geqslant \sum_{m\leqslant x^{1-\alpha}}\, \sum_{x^{\alpha}<p\leqslant x/m}
\frac{\omega(mp-1;\, x^{\gamma}, x^{\beta})}{(\frac{\log x}{\log x^{\beta}})}\cdot
\frac{\omega(mp+1;\, x^{\gamma}, x^{\beta})}{(\frac{\log x}{\log x^{\beta}})}
\nonumber\\\noalign{\vskip -2mm}
& = \beta^2\sum_{m\leqslant x^{1-\alpha}}\, \sum_{x^{\alpha}<p\leqslant x/m}
\sum_{\substack{x^{\beta}<p_1\leqslant x^{\gamma}\\ mp-1\equiv 0\, (\text{mod}\, p_1)}}
\, \sum_{\substack{x^{\beta}<p_2\leqslant x^{\gamma}\\ mp+1\equiv 0\, (\text{mod}\, p_2)}}1
\nonumber\\\noalign{\vskip 0mm}
& = \beta^2\sum_{m \leqslant x^{1-\alpha}}\,
\mathop{\sum\ \sum}_{\substack{x^{\beta} < p_1,\, p_2 \leqslant x^{\gamma}\\ p_1\neq p_2}}
\sum_{\substack{x^{\alpha}<p\leqslant x/m \\mp-1\equiv 0\,  (\textrm{mod}\, p_1)\\ mp+1\equiv 0\,  (\textrm{mod}\, p_2) }}1.
\end{align*}
D'apr\`{e}s \eqref{p1p2} et \eqref{b0p1p2}, on a donc
\begin{equation}\label{S1S2}
\begin{aligned}
\sum_{\substack{n\leqslant x\\ P^+(n-1)<P^+(n)>P^+(n+1)\\ x^{\alpha}<P^+(n)\leqslant x}}1
& \geqslant \beta^2 \sum_{m \leqslant x^{1-\alpha}} \,
\mathop{\sum\ \sum}_{\substack{x^{\beta} < p_1,\, p_2 \leqslant x^{\gamma}\\ p_1\neq p_2}}
\sum_{\substack{x^{\alpha}<p\leqslant x/m \\p\equiv b\, (\textrm{mod}\, p_1p_2) }}1
\\\noalign{\vskip 3mm}
& = \mathbb{S}_1+\mathbb{S}_2,
\end{aligned}
\end{equation}
o\`{u}
\begin{align*}
\mathbb{S}_1
& := \beta^2\sum_{m \leqslant x^{1-\alpha}} \,
\mathop{\sum\ \sum}_{\substack{x^{\beta} < p_1,\, p_2 \leqslant x^{\gamma}\\ p_1\neq p_2}}
\frac{\pi(x/m)-\pi(x^{\alpha})}{\varphi(p_1p_2)},
\\
\mathbb{S}_2
& := \beta^2 \sum_{m \leqslant x^{1-\alpha}} \,
\mathop{\sum\ \sum}_{\substack{x^{\beta} < p_1,\, p_2 \leqslant x^{\gamma}\\ p_1\neq p_2}}
\bigg\{\bigg(\pi(x/m;\, b, p_1p_2)-\frac{\pi(x/m)}{\varphi(p_1p_2)}\bigg)
\nonumber\\\noalign{\vskip 0mm}
& \hskip 45mm
- \bigg(\pi(x^{\alpha};\, b, p_1p_2)- \frac{\pi(x^{\alpha})}{\varphi(p_1p_2)} \bigg)\bigg\}.
\end{align*}

On majore $\mathbb{S}_2$ en utilisant le th\'{e}or\`{e}me de Bombieri-Vinogradov \cite{Bom65, Vin65} qui est le cas du Lemme \ref{lem:PanDingWang} avec $l=1$.
On impose une condition sur $\gamma$ et $\alpha$:
\begin{align}\label{cond:gamma2beta1}
2\gamma < \frac{\alpha}{2}\cdot
\end{align}
On en d\'{e}duit, en combinant les conditions \eqref{cond:beta1gamma1gamma2} et \eqref{cond:gamma2beta1}
\begin{align}
\frac{4}{5}<\alpha< 1,
\qquad
1-\alpha<\beta<\gamma<\frac{\alpha}{4}\cdot
\end{align}
En procédant de manière similaire \`{a} l'estimation de $\mathcal{S}_2$ dans \eqref{thm1:S2}, c'est-à-dire,  \`{a} l'aide de l'in\'{e}galit\'{e} de Cauchy-Schwarz et le th\'{e}or\`{e}me de Bombieri-Vinogradov, on obtient pour tout $A>0$
\begin{align}\label{thm3*:S2}
\mathbb{S}_2 \ll \frac{x}{(\log x)^{A}}.
\end{align}

Puis pour $\mathbb{S}_1$, on utilise le théorème des nombres premiers:
\begin{align*}
\mathbb{S}_1
& = \beta^2 \sum_{m \leqslant x^{1-\alpha}} \bigg(\pi\Big(\frac{x}{m}\Big)-\pi\big(x^{\alpha}\big)\bigg)
\,\mathop{\sum\ \sum}_{\substack{x^{\beta} < p_1,\, p_2 \leqslant x^{\gamma}\\ p_1\neq p_2}} \frac{1}{(p_1-1)(p_2-1)}
\\
& = \beta^2 \sum_{m \leqslant x^{1-\alpha}} \frac{x+o(x)}{m\log(\frac{x}{m})}
\,\mathop{\sum\ \sum}_{x^{\beta} < p_1,\, p_2 \leqslant x^{\gamma}}\frac{1}{p_1p_2}
+ O\bigg(\frac{x^{1-\beta}}{\log x} + \sum_{m \leqslant x^{1-\alpha}}\frac{x^{\alpha}}{\log x} \bigg)
\\
& = x \bigg(\beta \log\frac{\gamma}{\beta}\bigg)^2
\int^{x^{1-\alpha}}_{1}\frac{\text{d}t}{t(\log x-\log t)}+o(x)
\\\noalign{\vskip 2mm}
& = x \bigg(\beta \log\frac{\gamma}{\beta}\bigg)^2\log\frac{1}{\alpha}+o(x).
\end{align*}
\`{A} l'aide de \emph{Mathematica 9.0}, on obtient
\begin{align}\label{f(delta0)}
\widetilde{s_1}
:= \max_{\substack{\frac{4}{5}<\alpha< 1 \\ 1-\alpha<\beta<\gamma<\alpha/4 }}
\bigg(\beta \log\frac{\gamma}{\beta}\bigg)^2\log\frac{1}{\alpha}
= \bigg(\beta' \log\frac{\gamma'}{\beta'}\bigg)^2\log\frac{1}{\alpha'}
>7\times 10^{-4}
\end{align}
avec
$$
\alpha'\approx 0,895,
\qquad
\beta'\approx 0,105,
\qquad
\gamma'\approx 0,22375.
$$
Ainsi, on obtient
\begin{align}\label{0,895}
\sum_{\substack{n\leqslant x\\ P^+(n-1)<P^+(n)>P^+(n+1)\\ x^{0,895}<P^+(n)\leqslant x}}1> 7\times 10^{-4}x.
\end{align}
De mani\`{e}re similaire, on peut montrer que
\begin{align}\label{0,858}
\sum_{\substack{n\leqslant x\\ P^+(n-1)<P^+(n)>P^+(n+1)\\ x^{0,858}<P^+(n)\leqslant x^{0,895}}}1> 1,44\times 10^{-4}x
\end{align}
et
\begin{align}\label{0,835}
\sum_{\substack{n\leqslant x\\ P^+(n-1)<P^+(n)>P^+(n+1)\\ x^{0,835}<P^+(n)\leqslant x^{0,858}}}1> 4\times 10^{-5}x.
\end{align}
En combinant les minorantions \eqref{0,895}, \eqref{0,858} et \eqref{0,835}, on obtient finalement l'estimation
\begin{align*}
\sum_{\substack{n\leqslant x\\ P^+(n-1)<P^+(n)>P^+(n+1)\\ x^{0,835}<P^+(n)\leqslant x}}1> 8,84\times 10^{-4}x.
\end{align*}
Ce qui implique le r\'{e}sultat.

\vskip 8mm

\section{D\'{e}monstration du Th\'{e}or\`{e}me \ref{thm3}}

Sans perte de la g\'{e}n\'{e}ralit\'{e},
on peut supposer que $j_0=0,$ c'est-à-dire,
$$
P^+(n)=\min_{0\leqslant j\leqslant J-1} P^+(n+j).
$$
Soient $y := x^{\alpha}$ avec
$$
0<\alpha < \frac{1}{2(J-1)}\cdot
$$
La méthode est analogue à la preuve de \eqref{eq:thm1_A}. On considère les entiers $y-$friables $n$ tels que
$P^+(n+j)>y$:
\begin{align*}
\sum_{\substack{n\leqslant x\\ P^+(n)=\min\limits_{0\leqslant j\leqslant J-1}P^+(n+j)}} 1
& \geqslant \sum_{\substack{n\in S(x,\, y)\\\noalign{\vskip 0.5mm} (n+j, \, P(x, y))>1
\\\noalign{\vskip 0.5mm} \forall\, 1\leqslant j\leqslant J-1}} 1
\\\noalign{\vskip -1mm}
& \geqslant \sum_{n\in S(x,\, y)}\frac{\omega(n+1;\, x, y)}{(\frac{\log x}{\log y})}
\cdots \frac{\omega(n+J-1;\, x, y)}{(\frac{\log x}{\log y})}
\\\noalign{\vskip 1mm}
& \geqslant {\alpha}^{J-1}
\mathop{\sum_{y<p_1\le Y} \cdots \sum_{y<p_{J-1}\le Y}}_{p_i\not= p_j\;(i\not=j)}
\sum_{\substack{n\in S(x,\, y)\\ n\equiv -j\,(\text{mod}\, p_j) \\ j=1, 2, \, \ldots \, , J-1}} 1,
\end{align*}
o\`{u} $Y := (x^{1/2}/(\log x)^B)^{1/(J-1)}$.

D'après  le th\'{e}or\`{e}me des restes chinois,
il existe un entier $a$ tel que les congruences de la somme intérieur du membre de droite s'écrivent sous la forme  $n\equiv a (\textrm{mod}\ p_1\cdots p_{J-1})$
car $p_1, \ldots, p_{J-1}$ sont premiers entre eux.
On a donc
\begin{equation}\label{thm2:S3S4}
\begin{aligned}
\sum_{\substack{n\leqslant x\\ P^+(n)=\min\limits_{0\leqslant j\leqslant J-1}P^+(n+j)}}1
& \geqslant\, {\alpha}^{J-1} \mathop{\sum_{y<p_1\le Y} \cdots
\sum_{y<p_{J-1}\le Y}}_{p_i\not= p_j\;(i\not=j)}
\sum_{\substack{n\in S(x,\, y)\\ n\equiv a (\textrm{mod}\, p_1\cdots p_{J-1})}}1
\\\noalign{\vskip 3mm}
& = {\alpha}^{J-1} (\mathcal{S}_3+\mathcal{S}_4),
\end{aligned}
\end{equation}
o\`{u}
\begin{align*}
\mathcal{S}_3
& := \mathop{\sum_{y<p_1\le Y} \cdots \sum_{y<p_{J-1}\le Y}}_{p_i\not= p_j\;(i\not=j)}
\frac{\Psi_{p_1\cdots p_{J-1}}(x,\, y)}{\varphi(p_1\cdots p_{J-1})},
\\
\mathcal{S}_4
& := \mathop{\sum_{y<p_1\le Y} \cdots \sum_{y<p_{J-1}\le Y}}_{p_i\not= p_j\;(i\not=j)}
\bigg(\Psi(x, y;\, a, p_1\cdots p_{J-1})- \frac{\Psi_{p_1\cdots p_{J-1}}(x, y)}
{\varphi(p_1\cdots p_{J-1})}\bigg).
\end{align*}

La somme $\mathcal{S}_4$ se majore de la même fa\c{c}on que $\mathcal{S}_2$,
en utilisant l'in\'{e}galit\'{e} de Cauchy-Schwarz et le Lemme \ref{lem:BV-S(x,y)}:
\begin{align}\label{thm2:S4}
\mathcal{S}_4 \ll x(\log x)^{-A}.
\end{align}

Pour $\mathcal{S}_3$, la condition $p_i\neq p_j\, (i\neq j)$  peut être supprimée,
car la contribution des puissances de nombres premiers est néglisable
(puisque tous les $p_j> y$). On a donc
\begin{equation}\label{thm2:S3}
\begin{aligned}
\mathcal{S}_3
& = x \rho\bigg(\frac{1}{\alpha}\bigg) \mathop{\sum_{y_1<p_1\leqslant Y_1}\frac{1}{p_1}\cdots \sum_{y_1<p_{J-1}\leqslant Y_1}} \frac{1}{p_{J-1}} +o(x)
\\\noalign{\vskip 1mm}
& = x \rho\bigg(\frac{1}{\alpha}\bigg) \bigg(\log\frac{1}{2\alpha(J-1)}\bigg)^{J-1}  +o(x).
\end{aligned}
\end{equation}

On peut donc d\'{e}duire que,  par \eqref{thm2:S3S4} \eqref{thm2:S4} et \eqref{thm2:S3}
\begin{align*}
\sum_{\substack{n\leqslant x\\ P^+(n)=\min\limits_{0\leqslant j\leqslant J-1}P^+(n+j)}} 1
& \geqslant \{f(J, \alpha) x + o(1)\} x
\end{align*}
où
$$
f(J, \alpha) := \rho(\alpha^{-1}) (-\alpha\log(2\alpha(J-1)))^{J-1}.
$$
La fonction $\alpha\mapsto f(J, \alpha)$ est continue et strictement positive sur $]0, \frac{1}{2(J-1)}[$ et
\begin{align*}
\lim_{\alpha\rightarrow 0^+} f(J, \alpha)=0,
\qquad
\lim_{\alpha\rightarrow (2(J-1))^{-1}-} f(J, \alpha)=0.
\end{align*}
Il existe donc un $\alpha_0\in\,]0, \frac{1}{2(J-1)}[$ tel que
$$
C_3(J) :=\max_{0<\alpha < \frac{1}{2(J-1)}} f(J, \alpha) = f(J, \alpha_0)>0.
$$
Ce qui ach\`{e}ve la d\'{e}monstration du Th\'{e}or\`{e}me \ref{thm3}.

\vskip 8mm

\section{D\'{e}monstration du Th\'{e}or\`{e}me \ref{thm4}}

Comme \eqref{eq:thm1_A}  du Th\'{e}or\`{e}me \ref{thm1} et le Th\'{e}or\`{e}me \ref{thm3},
la d\'{e}monstration du Th\'{e}or\`{e}me \ref{thm4} suit le même fil que
celle de \eqref{eq:thm1_B} du Th\'{e}or\`{e}me \ref{thm1}.

Pour $J\in \mathbb{Z}$, $J\geqslant 3$, on peut supposer,
sans perte de g\'{e}n\'{e}ralit\'{e}, que $j_0=0$, i.e.
$$
P^+(n)=\max_{0\leqslant j\leqslant J-1}P^+(n+j).
$$
Soient $\alpha, \beta, \gamma$ tels que
\begin{equation}\label{cond:alphabetagamma}
\frac{2J-2}{2J-1}<\alpha<1
\qquad\text{et}\qquad
\quad 1-\alpha\leqslant \beta< \gamma<\frac{\alpha}{2(J-1)}
\end{equation}
trois paramètres à choisir plus tard.
\'Etant donnés un entier $m$ et $J-1$ nombres premiers distincts $p_1, \dots, p_{J-1}$ vérifiant
\begin{equation}\label{cond:p1p2pJ-1}
1\le m\le x^{1-\alpha}
\qquad\text{et}\qquad
x^{\beta}<p_1,\, p_2,\, \dots,\, p_{J-1}\leqslant x^{\gamma},
\end{equation}
on considère le système d'équations de congruences :
\begin{align}\label{p1p2pJ-1}
\begin{split}
mp+1 & \equiv0\, (\text{mod}\, p_1),
\\\noalign{\vskip 0mm}
mp+2 & \equiv0\, (\text{mod}\, p_2),
\\\noalign{\vskip 0mm}
& \hskip 2,5mm\vdots
\\\noalign{\vskip 0mm}
mp+J-1 & \equiv0\, (\text{mod}\, p_{J-1}).
\end{split}
\end{align}
En remarquant que la condition \eqref{cond:alphabetagamma} garantit que $p_j\nmid m$
pour $1\le j\le J-1$, le th\'{e}or\`{e}me chinois et \eqref{p1p2pJ-1},
il existe $b<p_1p_2\cdots p_{J-1}$ tel que \eqref{p1p2pJ-1} soit équivalent à
\begin{align*}
p\equiv b\, (\text{mod}\, p_1p_2\cdots p_{J-1}).
\end{align*}
Pour ces $m$ et $p>x^{\alpha}$, il est facile de voir que
$$
P^+(mp)=\max\limits_{0\leqslant j\leqslant J-1}P^+(mp+j).
$$
On a donc
\begin{align*}
\sum_{\substack{n\leqslant x\\ P^+(n)=\max\limits_{0\leqslant j\leqslant J-1}P^+(n+j)}}1
& \geqslant \sum_{m\leqslant x^{1-\alpha}}\, \sum_{x^{\alpha}<p\leqslant x/m}\,
\prod_{1\le j\le J-1}
\frac{\omega(mp+j; \, x^{\gamma}, x^{\beta})}{(\frac{\log x}{\log x^{\beta}})}
\\
& = \beta^{J-1}\sum_{m\leqslant x^{1-\alpha}}
\mathop{\sum\;\cdots\;\sum}_{\substack{x^{\beta}<p_1,\, \dots,\, p_{J-1}\leqslant x^{\gamma}\\
\text{$p_1,\, \dots,\, p_{J-1}$ sont distincts}}} \;
\sum_{\substack{x^{\alpha}<p\leqslant x/m\\ p\equiv b\, (\text{mod}\, p_1p_2\cdots p_{J-1})}} 1
\end{align*}
Comme pour \eqref{thm2:S3}, la condition que $p_1,\, \dots,\, p_{J-1}$ soient distincts peut être supprimée puisque tous les $p_j> x^{\beta}$. En remarquant que $p_1\cdots p_{J-1}\le x^{\alpha/2}$,
on peut appliquer le théorème de Bombieri-Vinogradov (voir le Lemme \ref{lem:PanDingWang} ci-dessus)
\begin{align}
\sum_{\substack{n\leqslant x\\ P^+(n)=\max\limits_{0\leqslant j\leqslant J-1}P^+(n+j)}}1
& \geqslant \beta^{J-1}\sum_{m\leqslant x^{1-\alpha}}
\mathop{\sum\;\cdots\;\sum}_{x^{\beta}<p_1,\, \dots,\, p_{J-1}\leqslant x^{\gamma}}
\frac{\pi(x/m)-\pi(x^{\alpha})}{\varphi(p_1p_2\cdots p_{J-1})}+O\Big(\frac{x}{(\log x)^A}\Big)
\nonumber\\\noalign{\vskip 0mm}
& = x \bigg(\beta\log\frac{\gamma}{\beta}\bigg)^{J-1}\log\frac{1}{\alpha}+o(x).
\end{align}
Il est clair que la fonction continue
$(\alpha, \beta, \gamma)\mapsto \big(\beta\log\frac{\gamma}{\beta}\big)^{J-1}\log\frac{1}{\alpha}$
est strictement positive et on peut prendre
$$
C_4(J)
:= \max_{(\alpha, \beta, \gamma)\;\text{vérifient \eqref{cond:alphabetagamma}}}
\bigg(\beta\log\frac{\gamma}{\beta}\bigg)^{J-1}\log\frac{1}{\alpha}>0,
$$
ce qui ach\`{e}ve la d\'{e}monstration du Th\'{e}or\`{e}me \ref{thm4}.

\vskip 8mm

\section{D\'{e}monstration du Th\'{e}or\`{e}me \ref{thm2}(i)}

Soient $y=x^{\alpha}$ et $z=x^{\beta}$ avec $0<\beta<\alpha\leqslant \frac{1}{2}$.
Comme précédemment, on peut écrire
\begin{equation}\label{thm31:S1S2}
\begin{aligned}
\sum_{\substack{n\leqslant x\\ P^+_y(n)<P^+_y(n+1)}} 1
& \geqslant \sum_{\substack{n\in S(x;\, y, z)\\ \noalign{\vskip 1mm} (n+1,\ P(y, z))>1}}1
\\\noalign{\vskip 0mm}
& \geqslant \sum_{n\in S(x;\, y, z)}\frac{\omega(n+1;\, y, z)}{(\tfrac{\log x}{\log z})}
\\\noalign{\vskip 0mm}
& =\beta \sum_{z<p\leqslant y} \sum_{\substack{n\in S(x;\, y, z)\\n\equiv -1\, (\textrm{mod}\, p)}}1
\\\noalign{\vskip 0mm}
& = \beta (S_1+S_2),
\end{aligned}
\end{equation}
o\`{u}
\begin{align*}
S_1
& := \sum_{z<p\leqslant y} \frac{1}{\varphi(p)}\sum_{\substack{n\in S(x;\, y, z)\\(n,\, p)=1}}1
\\
S_2
& := \sum_{z<p\leqslant y}
\bigg(\sum_{\substack{n\in S(x;\, y, z)\\n\equiv -1\, (\textrm{mod}\, p)}} 1
- \frac{1}{\varphi(p)}\sum_{\substack{n\in S(x;\, y, z)\\(n,\, p)=1}} 1
\bigg).
\end{align*}
Pour $S_1$, en appliquant le Lemme \ref{lem:S(x;y,z)} avec $u=1/\alpha$ et $\lambda=\beta/\alpha$,
on obtient pour $\alpha+\beta \leqslant 1$
\begin{equation}\label{thm31:S1}
\begin{aligned}
S_1
& = \sum_{z<p\leqslant y} \frac{1}{\varphi(p)}\sum_{n\in S(x;\, y, z)}1
\\\noalign{\vskip 0mm}
& = \sum_{z<p\leqslant y} \frac{1}{p-1} \cdot x \vartheta_0\bigg(\frac{\beta}{\alpha}, \frac{1}{\alpha}\bigg)
\bigg\{1+O\bigg(\frac{1}{\log z}\bigg)\bigg\}
\\\noalign{\vskip 1mm}
&= x \vartheta_0\bigg(\frac{\beta}{\alpha}, \frac{1}{\alpha}\bigg) \log\bigg(\frac{\alpha}{\beta}\bigg)+o(x).
\end{aligned}
\end{equation}

Pour $S_2$, si $y\leqslant x^{1/2}/(\log x)^B$, o\`{u} $B=B(A)$ est une grande constante, on a en appliquant la Proposition \ref{prop:S(x;y,z)}
\begin{align}\label{thm31:S2*}
S_2 \ll \frac{x}{(\log x)^A}.
\end{align}
De plus, si $x^{1/2}/(\log x)^B< y\leqslant x^{1/2}$, une estimation triviale permet d'obtenir
\begin{equation}\label{thm31:S2**}
\begin{aligned}
S_2 & \ll \sum_{x^{1/2}/(\log x)^B< p \leqslant x^{1/2}}
\bigg|
\sum_{\substack{n\in S(x;\, y, z)\\ n\equiv -1\, (\textrm{mod}\, p)}} 1
- \frac{1}{\varphi(p)} \sum_{\substack{n\in S(x;\, y, z)\\(n,\, p)=1}}1 \bigg|
+ \frac{x}{(\log x)^A}
\\\noalign{\vskip 1mm}
&\ll x\sum_{x^{1/2}/(\log x)^B< p \leqslant x^{1/2}}\frac{1}{p} + \frac{x}{(\log x)^A}
\\\noalign{\vskip 1mm}
&\ll \frac{x\log_2x}{\log x}\cdot
\end{aligned}
\end{equation}

En reportant \eqref{thm31:S1}, \eqref{thm31:S2*} et \eqref{thm31:S2**} dans \eqref{thm31:S1S2},
on obtient pour $0<\alpha\leqslant \frac{1}{2}$
\begin{align}\label{thm31}
\sum_{\substack{n\leqslant x\\ P^+_y(n)<P^+_y(n+1)}}1
& \geqslant C(\alpha)x,
\end{align}
o\`{u} $C(\alpha)$ est d\'{e}finie par
\begin{align}\label{def:C(alpha)(i)}
C(\alpha)
:= \max_{0<\beta<\alpha}
\vartheta_0\bigg(\frac{\beta}{\alpha}, \frac{1}{\alpha}\bigg)\beta\log\bigg(\frac{\alpha}{\beta}\bigg)>0 \qquad(0<\alpha\leqslant 1/2)
\end{align}
et $\vartheta_0$ est d\'{e}finie par \eqref{def:vartheta0}.
Ce qui ach\`{e}ve la d\'{e}monstration du Th\'{e}or\`{e}me \ref{thm2}(i).

\vskip 8mm

\section{D\'{e}monstration du Th\'{e}or\`{e}me \ref{thm2}(ii)}

\subsection{Point de départ}\

\vskip 1mm

Pour $y=x^{\alpha}$ avec $\frac{1}{2}<\alpha\leqslant 1$ et $1-\alpha\le c\le \frac{1}{2}$, on peut écrire
\begin{equation}\label{SSASBSC}
\begin{aligned}
\sum_{\substack{n\leqslant x\\ P^+_y(n)<P^+_y(n+1)}} 1
& = \sum_{\substack{n\leqslant x\\ y\geqslant P^+(n+1)>x^{1-c}}} 1
- \sum_{\substack{n\leqslant x\\ y\geqslant P^+(n)>P^+(n+1)>x^{1-c}}} 1
+ \sum_{\substack{n\leqslant x\\ P^+_y(n)<P^+_y(n+1)\leqslant x^{1-c}}} 1
\\\noalign{\vskip 2mm}
& =: \mathscr{S}_A-\mathscr{S}_B+\mathscr{S}_C.
\end{aligned}
\end{equation}
Il reste \`{a} maintenant minorer $\mathscr{S}_A$, $\mathscr{S}_C$ et majorer $\mathscr{S}_B.$

\subsection{Estimation de la somme $\mathscr{S}_A$}\

\vskip 1mm

Pour $\mathscr{S}_A$, d'apr\`{e}s le Lemme \ref{lem2.3}, on a
\begin{equation}\label{SA}
\begin{aligned}
\mathscr{S}_A
& = x\bigg\{\rho\bigg(\frac{\log x}{\log y}\bigg)-\rho\bigg(\frac{\log x}{\log x^{1-c}}\bigg)\bigg\}+o(x)
\\\noalign{\vskip 1mm}
& = x\bigg\{\rho\bigg(\frac{1}{\alpha}\bigg)-\rho\bigg(\frac{1}{1-c}\bigg)\bigg\}+o(x)
\\\noalign{\vskip 1mm}
& = x\log\bigg(\frac{\alpha}{1-c}\bigg)+o(x)
\end{aligned}
\end{equation}
pour $\frac{1}{2}\leqslant 1-c\le \alpha \leqslant 1$,
o\`{u} on a utilis\'{e} le fait que $\rho(u)=1-\log u \; (1\leqslant u\leqslant 2)$.

\subsection{Estimation de la somme $\mathscr{S}_B$}\

\vskip 1mm

Pour $\mathscr{S}_B$, on utilise le crible de Rosser-Iwaniec comme l'auteur l'a fait dans \cite{Wang17}. On observe tout d'abord que, $\mathscr{S}_B$ compte le nombre des entiers $n$ tels que $n=ap=bp'-1,$ avec $x^{1-c}<p'<p\leqslant y.$
On a ainsi
\begin{align*}
\mathscr{S}_B
& \leqslant \big|\big\{n\leqslant x:\ n=ap=bp'-1,\ x^{1-\alpha}<a<b\leqslant x^c \big\}\big|+o(x)
\\\noalign{\vskip 1mm}
& \leqslant \sum_{x^{1-\alpha}<b\leqslant x^c} \big|\big\{n\in\mathscr{A}(b) : n\ \text{est premier}\big\}\big| +o(x),
\end{align*}
o\`{u}
$$
\mathscr{A}(b)
:= \bigg\{\frac{ap+1}{b}:\, ap\leqslant x,\, x^{1-\alpha}<a< b,\, ap+1\equiv0\,(\textrm{mod}\, b) \bigg\}.
$$
On prend $\mathcal{P} = \{p:\ p \ \text{est premier}\}$,
et on d\'{e}duit que, pour $b\leqslant x^c$ et $d\mid P(z):=\prod_{p<z} p$,
o\`{u} $z$ sera explicit\'{e} plus tard,
\begin{equation}\label{Ad(b)}
\begin{aligned}
|\mathscr{A}_d(b)|
& =\bigg|\bigg\{\frac{ap+1}{b}:\, ap\leqslant x,\ x^{1-\alpha}<a<b, ap+1\equiv 0 \,(\bmod\, bd) \bigg\}\bigg|
\\\noalign{\vskip 1mm}
& = \sum_{\substack{x^{1-\alpha}<a<b\\ (a,\, bd)=1}}\ \frac{\text{li}\, (x/a)}{\varphi(bd)}
+ \sum_{\substack{x^{1-\alpha}<a<b\\ (a,\, bd)=1}} E(x/a;\, -\overline{a}, bd),
\end{aligned}
\end{equation}
o\`{u}
\begin{align*}
E(x;\, a, q) := \pi(x;\, a, q)-\frac{\text{li}\, x}{\varphi(q)},
\qquad
a\overline{a}\equiv 1\,({\rm mod}\,bd).
\end{align*}

Pour la premi\`{e}re somme du membre de droit de \eqref{Ad(b)}, en utilisant le th\'{e}or\`{e}me des nombres premiers et
la formule d'inversion de M\"{o}bius, on a
\begin{align*}
\sum_{\substack{x^{1-\alpha}<a<b\\ (a,\, bd)=1}}\ \frac{\text{li}\, (x/a)}{\varphi(bd)}
& = \frac{1}{\varphi(bd)}\sum_{\substack{x^{1-\alpha}<a<b\\ (a,\, bd)=1}}\ \frac{x}{a\log(x/a)}\bigg\{1+O\bigg(\frac{1}{\log x}\bigg)\bigg\}
\nonumber\\\noalign{\vskip 1mm}
& = \big\{1+o(1)\big\} \frac{x}{\varphi(bd)}
\sum_{x^{1-\alpha}<a< b} \frac{1}{a\log(x/a)} \sum_{q\mid (a,\, bd)} \mu(q)
\nonumber\\\noalign{\vskip 1mm}
& = \big\{1+o(1)\big\} \frac{x}{\varphi(bd)}
\sum_{q\mid bd}\frac{\mu(q)}{q}\sum_{x^{1-\alpha}<aq< b}\frac{1}{a\log(x/aq)}\cdot
\end{align*}
Par une int\'{e}gration par parties, on obtient
\begin{equation}\label{Ad(b):main}
\begin{aligned}
\sum_{\substack{x^{1-\alpha}<a<b\\ (a,\, bd)=1}}\ \frac{\text{li}\, (x/a)}{\varphi(bd)}
& = \big\{1+o(1)\big\} \frac{x}{\varphi(bd)} \sum_{q\mid bd} \frac{\mu(q)}{q}
\bigg\{\int^{b/q}_{x^{1-\alpha}/q} \frac{{\rm{d}}t}{t\log(x/tq)}+o(1)\bigg\}
\\\noalign{\vskip -1mm}
& = \bigg\{\log\bigg(\frac{\alpha\log x}{\log(x/b)}\bigg) + o(1)\bigg\} \frac{x}{\varphi(bd)}
\sum_{q\mid bd} \frac{\mu(q)}{q}
\\\noalign{\vskip 1mm}
& = \bigg\{\log\bigg(\frac{\alpha\log x}{\log(x/b)}\bigg) + o(1)\bigg\}\, \frac{x}{bd}\cdot
\end{aligned}
\end{equation}

Compte tenu de \eqref{Ad(b)} et \eqref{Ad(b):main}, on en d\'{e}duit  que
$$
|\mathscr{A}_d(b)| = \frac{w(d)}{d} X + r\big(\mathscr{A}(b),d \big)
$$
avec
$$
w(d)=1,
\quad
X = \frac{x}{b}\, \bigg\{\log\bigg(\frac{\alpha\log x}{\log(x/b)}\bigg) + o(1)\bigg\},
\quad
r\big(\mathscr{A}(b),d \big)
= \sum_{\substack{x^{1-\alpha}<a<b\\ (a,\, bd)=1}} E(x/a;\, -\overline{a}, bd).
$$
On peut donc appliquer le Lemme \ref{lem:sieve} avec $D = z^2 = x^{4/7-\varepsilon}/b$
et obtenir
$$
S(\mathscr{A}(b); \mathcal{P}, z)
\leqslant \{1+o(1)\} \frac{2X}{\log(x^{4/7-\varepsilon}/b)}
+ \sum_{\ell<L}\, \sum_{d<D, \, d\mid P(z)}\lambda_{\ell}^+(d)\, r\big(\mathscr{A}(b),d \big),
$$
o\`{u} $L:= \exp(8/\varepsilon^3)$ et on a utilis\'{e} la formule de Mertens
$$
V(z)
= \prod_{p<z}\bigg(1-\frac{1}{p}\bigg)
= \frac{2\mathrm{e}^{-\gamma}}{\log(x^{4/7-\varepsilon}/b)}\big\{1+o(1)\big\}.
$$
Compte tenu de l'\'{e}valuation
$$
\big|\{n\in \mathscr{A}(b):\, n\ \text{est premier} \}\big|\leqslant S(\mathscr{A}(b); \mathcal{P}, z)+z,
$$
on obtient
\begin{equation}\label{SBSB1SB2}
\begin{aligned}
\mathscr{S}_B
& \leqslant \sum_{x^{1-\alpha}<b\leqslant x^c} \big( S(\mathscr{A}(b); \mathcal{P}, z)+z\big)
\\\noalign{\vskip 1mm}
& \leqslant \big\{1+o(1)\big\} \mathscr{S}_{B1} + \mathscr{S}_{B2} + O\big(x(\log x)^{-1}\big),
\end{aligned}
\end{equation}
o\`{u}
\begin{align*}
\mathscr{S}_{B1}
& := \sum_{x^{1-\alpha}<b\leqslant x^c} \frac{2x}{b\log(x^{4/7-\varepsilon}/b)}
\log\bigg(\frac{\alpha\log x}{\log(x/b)}\bigg),
\\\noalign{\vskip 1mm}
\mathscr{S}_{B2}
& := \sum_{\ell<L}\, \sum_{x^{1-\alpha}<b\leqslant x^c}\,
\sum_{\substack{d<D \\ d\mid P(z)}}\lambda_{\ell}^+(d)\, r\big(\mathscr{A}(b),d \big).
\end{align*}

Nous allons majorer $\mathscr{S}_{B2}$ à l'aide de la Proposition \ref{prop:4/7}.
Pour cela, écrivons d'abord,
\begin{align*}
\mathscr{S}_{B2}
& = \sum_{\ell<L} \sum_{\mathcal{B}} \sum_{\mathcal{B}/2<b\le \mathcal{B}}
\sum_{\substack{d<D \\ d\mid P(z)}}\lambda_{\ell}^+(d)\, \bigg\{ \sum_{\substack{x^{1-\alpha}<a<b\\ (a,\, bd)=1}}
\Big( \pi(x;\, a, -1, bd) -\frac{{\rm{li}}\,(x/a)}{\varphi(bd)} \Big)\bigg\}.
\end{align*}
De plus, on impose une condition  sur $c$

$$
c\leqslant 2/7-\varepsilon
$$
telle que
$$
x^{c}<D = x^{4/7-\varepsilon}/b.
$$

D\'{e}signons la fonction $\lambda_{\ell}$ par
\begin{align}\label{def:lambda}
\lambda_{\ell}(q)
:= \mathop{\sum_{\mathcal{B}/2 <b\leqslant \mathcal{B}}\, \sum_{d<D, \, d\mid P(z)}}_{bd=q}
\mathbb{1}_{]\mathcal{B}/2,\, \mathcal{B}]}(b)\, \lambda_{\ell}^+(d),
\end{align}
o\`{u}
\begin{align*}
\mathbb{1}_{]\mathcal{B}/2,\, \mathcal{B}]}(b)=\left\{
\begin{array}{ll}
    1 & \quad  \textmd{si} \ \ \mathcal{B}/2 <b\leqslant \mathcal{B},
    \\\noalign{\vskip 1mm}
    0 &  \quad  \textmd{sinon}.
  \end{array}
\right.
\end{align*}
Pour utiliser la Proposition 2,
il faut d\'{e}montrer que $\lambda_{\ell}(q)$ est bien factorisable de niveau $Q=\mathcal{B}D.$
Pour toute d\'{e}composition $x^{4/7-\varepsilon} = \mathcal{B}D=Q_1Q_2,$ $Q_1, Q_2\geqslant 1$,
compte tenu de $\mathcal{B}\leqslant D,$ il existe un $i\in \{1,  2\}$
tel que $Q_i\le x^{2/7-\varepsilon}\le D$.
Sans perte de la g\'{e}n\'{e}ralit\'{e}, on peut supposer  que $Q_1\leqslant D$.
Puisque $\lambda_{\ell}^+$ est une fonction bien factorisable de niveau $D$,
il existe deux fonctions arithm\'{e}tiques
$\lambda_1^+$ de niveau $Q_1$ et $\lambda_2^+$ de niveau $D/Q_1$ telles que
$$
\lambda_{\ell}^+=\lambda_1^+ \ast \lambda_2^+.
$$
Ainsi, on peut trouver deux fonctions arithm\'{e}tiques $\lambda_1^+$ de niveau $Q_1$ et
$\lambda_2^+\ast \mathbb{1}_{]\mathcal{B}/2,\, \mathcal{B}]}$ de niveau $Q_2=(D/Q_1)\cdot \mathcal{B}$ telles que
$$
\lambda
=\lambda_1^+ \ast \big(\lambda_2^+\ast \mathbb{1}_{]\mathcal{B}/2,\, \mathcal{B}]} \big).
$$
Cela montre que $\lambda$ est une fonction bien factorisable
de niveau $\mathcal{B}D = x^{4/7-\varepsilon}$.
En appliquant la Proposition \ref{prop:4/7} avec $\lambda(q)$ d\'{e}finie par \eqref{def:lambda} et
$(L_1, L_2, a, \ell )=(x^{1-\alpha}, b, -1, a)$, on obtient
$$
\sum_{\mathcal{B}/2 <b\leqslant \mathcal{B}}\,
\sum_{\substack{d<D \\ d\mid P(z)}}\lambda_{\ell}^+(d) \,
\bigg\{ \sum_{\substack{x^{1-\alpha}<a<b\\ (a,\, bd)=1}}
\Big( \pi(x;\, a, -1, bd) -\frac{{\rm{li}}\,(x/a)}{\varphi(q)} \Big)\bigg\}
\ll \frac{x}{(\log x)^2}\cdot
$$
D'où
\begin{equation}\label{SB2}
\mathscr{S}_{B2}\ll \frac{x}{\log x}\cdot
\end{equation}

Pour $\mathscr{S}_{B1}$, en utilisant une int\'{e}gration par parties on a
\begin{equation}\label{SB1}
\begin{aligned}
\mathscr{S}_{B1}
& = 2x \{1+o(1)\}
\int_{x^{1-\alpha}}^{x^c}
\log\bigg(\frac{\alpha\log x}{\log(x/t)}\bigg) \frac{\text{d} t}{t\log(x^{4/7-\varepsilon}/t)}
\\\noalign{\vskip 2mm}
& = x\bigg\{2\int_{1-\alpha}^{c} \log\bigg(\frac{\alpha}{1-t}\bigg)\frac{{\rm{d}}t}{\frac{4}{7}-t}
+ o_{c,\, \varepsilon, \,  \alpha}(1)\bigg\}.
\end{aligned}
\end{equation}

En reportant \eqref{SB2} et \eqref{SB1} dans \eqref{SBSB1SB2},  on arrive à l'inégalité
\begin{align}\label{SB4/7}
\mathscr{S}_B
\leqslant x\bigg\{2\int_{1-\alpha}^{c}\log\bigg(\frac{\alpha}{1-t}\bigg)\frac{{\rm{d}}t}{\frac{4}{7}-t}
+o_{c,\, \varepsilon, \,  \alpha}(1)\bigg\}
\end{align}
avec
\begin{align*}
0<c \leqslant \tfrac{2}{7}-\varepsilon, \quad 1-c\le \alpha \le 1.
\end{align*}

De mani\`{e}re similaire,
en utilisant le Lemme \ref{lem:PanDingWang} à la place de la Proposition \ref{prop:4/7},
on peut montrer
\begin{align}\label{SB1/2}
\mathscr{S}_B
\leqslant x\bigg\{2\int_{1-\alpha}^{c}\log\bigg(\frac{\alpha}{1-t}\bigg)\frac{{\rm{d}}t}{\frac{1}{2}-t}
+o_{c,\, \varepsilon, \,  \alpha}(1)\bigg\}
\end{align}
avec
\begin{align*}
\tfrac{2}{7}-\varepsilon< c < \tfrac{1}{2},
\qquad
1-c\le \alpha \le 1.
\end{align*}

\subsection{Estimation de la somme $\mathscr{S}_C$}\

\vskip 1mm

En vue de traiter $\mathscr{S}_C$, on a tout d'abord pour $B$ assez grand
\begin{align*}
\mathscr{S}_C :=\sum_{\substack{n\leqslant x\\ P^+_y(n)<P^+_y(n+1)\leqslant x^{1-c}}}1
& \geqslant \sum_{\substack{n\in S(x;\, y,\, x^{\delta})\\ (n+1,\, P(x^{1-c},\, x^{\delta}))>1}}1
\nonumber\\ \noalign{\vskip 3mm}
& \geqslant  \sum_{\substack{n\in S(x;\, x^{\alpha},\, x^{\delta})\\
(n+1,\, P(x^{1/2}/(\log x)^B,\, x^{\delta}))>1}} 1
\end{align*}
o\`{u} $\frac{1}{2}\leqslant 1-c<\alpha \leqslant 1$, $\delta$ est un param\`{e}tre satisfaisant $c\leqslant \delta\leqslant \frac{1}{2}$.
On rappelle la d\'{e}finition de $\omega(n+1;\, y,\, y_1)$ dans \eqref{def:omega} et obtient
$$
\omega\big(n+1;\, x^{1/2}/(\log x)^B,\, x^{\delta}\big)
\leqslant \frac{\log x}{\log x^{\delta}}
= \frac{1}{\delta},
$$
d'o\`{u}
\begin{align}\label{omega}
\delta\, \omega\bigg(n+1;\, x^{1/2}/(\log x)^B,\, x^{\delta}\bigg)
\left\{
\begin{array}{ll}
    \leqslant 1 &  \textmd{si} \ \ \big(n+1,\, \prod\limits_{x^{\delta}<p\leqslant x^{1/2}/(\log x)^B}p\big)>1,
    \\
    =0 &   \textmd{sinon} .
  \end{array}
\right.
\end{align}
On a donc
\begin{equation}\label{SCSC1SC2}
\begin{aligned}
\mathscr{S}_C
& \geqslant
\delta \sum_{n\in S(x;\, x^{\alpha},\, x^{\delta})} \omega\big(n+1;\, x^{1/2}/(\log x)^B,\, x^{\delta}\big)
\\
& = \delta \sum_{n\in S(x;\, x^{\alpha},\, x^{\delta})}
\sum_{\substack{p\mid n+1\\ x^{\delta}<p\leqslant x^{1/2}/(\log x)^B}} 1
\\
& =\delta \sum_{x^{\delta}<p\leqslant x^{1/2}/(\log x)^B}
\sum_{\substack{n\in S(x;\, x^{\alpha},\, x^{\delta})\\n\equiv -1 (\textrm{mod}\, p)}}1
\\
& = \delta (\mathscr{S}_{C1}+\mathscr{S}_{C2}),
\end{aligned}
\end{equation}
avec
\begin{align*}
\mathscr{S}_{C1}
& := \sum_{x^{\delta}<p\leqslant x^{1/2}/(\log x)^B}
\frac{1}{\varphi(p)}\sum_{\substack{n\in S(x;\, x^{\alpha},\, x^{\delta})\\(n,\, p)=1}} 1,
\\
\mathscr{S}_{C2}
& := \sum_{x^{\delta}<p\leqslant x^{1/2}/(\log x)^B}
\bigg(\sum_{\substack{n\in S(x;\, x^{\alpha},\, x^{\delta})\\n\equiv -1\, (\textrm{mod}\, p)}}1
-\frac{1}{\varphi(p)}\sum_{\substack{n\in S(x;\, x^{\alpha},\, x^{\delta})\\(n,\, p)=1}}1 \bigg).
\end{align*}

Pour \'{e}valuer $\mathscr{S}_{C2}$, on d\'{e}duit de la Proposition \ref{prop:S(x;y,z)} que
\begin{equation}\label{SC2}
\begin{aligned}
\mathscr{S}_{C2}
& \ll \sum_{q\leqslant x^{1/2}/(\log x)^B}
\bigg|\sum_{\substack{n\in S(x;\, x^{\alpha},\, x^{\delta})\\n\equiv -1\, (\textrm{mod}\, q)}}1
-\frac{1}{\varphi(q)}\sum_{\substack{n\in S(x;\, x^{\alpha},\, x^{\delta})\\(n,\, q)=1}}1\bigg|
\\
& \ll x(\log x)^{-A}
\end{aligned}
\end{equation}
pour tout $A>0$.

Il reste \`{a} \'{e}valuer $\mathscr{S}_{C1}$. On utilise le Lemme \ref{lem:S(x;y,z)} sur les entiers
sans facteur premier dans un intervalle donné.
\begin{equation}\label{SC1}
\begin{aligned}
\mathscr{S}_{C1}
&=  \sum_{x^{\delta}<p\leqslant x^{1/2}/(\log x)^B}
\frac{1}{p-1}\cdot \Psi_0(x;\, x^{\alpha},\, x^{\delta})
\\\noalign{\vskip 3mm}
&= x \vartheta_0\bigg(\frac{\delta}{\alpha},\, \frac{1}{\alpha}\bigg) \log\frac{1}{2\delta}+o(x).\end{aligned}
\end{equation}

D'apr\`{e}s \eqref{SCSC1SC2}, \eqref{SC2} et \eqref{SC1},
on a ainsi pour $c\leqslant\delta\leqslant1/2$
\begin{align}\label{SC}
\qquad \mathscr{S}_{C}
\geqslant x\vartheta_0\bigg(\frac{\delta}{\alpha},\, \frac{1}{\alpha}\bigg) \delta\log\frac{1}{2\delta}+o(x).
\end{align}

En reportant \eqref{SA}, \eqref{SB4/7}, \eqref{SB1/2} et \eqref{SC} dans \eqref{SSASBSC},
on obtient finalement pour $\frac{1}{2}<\alpha \leqslant 1$
\begin{align*}
\sum_{\substack{n\leqslant x\\ P^+_y(n)<P^+_y(n+1)}}1 \geqslant g(\alpha; c, \delta)x+o(x)
\end{align*}
o\`{u} $g(\alpha; c, \delta)$ est définie par
\begin{equation}\label{def:galphacdelta}
g(\alpha; c, \delta)
:= \log\bigg(\frac{\alpha}{1-c}\bigg)
- 2 \int_{1-\alpha}^{c} \log\bigg(\frac{\alpha}{1-t}\bigg) \frac{{\rm{d}}t}{\nu(c)-t}
+ \vartheta_0\bigg(\frac{\delta}{\alpha}, \frac{1}{\alpha}\bigg) \delta\log\frac{1}{2\delta}
\end{equation}
et
\begin{equation}\label{def:nuc}
\nu(c) :=
\begin{cases}
\tfrac{4}{7} & \text{si $\;0<c\leqslant \tfrac{2}{7}-\varepsilon$},
\\\noalign{\vskip 1mm}
\frac{1}{2}  &  \text{si $\;\tfrac{2}{7}-\varepsilon <c<\tfrac{1}{2}$}.
\end{cases}
\end{equation}

Il est clair que $g(\alpha; c, \delta)>0$ pour $1/2<\alpha\leqslant 1$.
On peut v\'{e}rifier de la m\^{e}me manière que lors de la preuve du Th\'{e}or\`{e}me 3, qu'il existe des $c_0$, $\delta_0$ satisfaisant les conditions ci-dessus
telles que pour $\frac{1}{2}<\alpha\leqslant 1$
\begin{align}\label{def:C(alpha)(ii)}
C(\alpha)
:= \max_{\frac{1}{2}\leqslant 1-c<\alpha}\, \max_{c\leqslant\delta\leqslant\frac{1}{2}}
g(\alpha; c, \delta)=g(\alpha; c_0, \delta_0) >0 \qquad (1/2<\alpha\leqslant 1),
\end{align}
ce qui ach\`{e}ve la d\'{e}monstration du Th\'{e}or\`{e}me \ref{thm2}(ii).

On a pu prendre ici le niveau $\nu(c)=4/7$ avec $c\leqslant 2/7-\varepsilon$ dans la somme $\mathscr{S}_B$
de \eqref{SSASBSC} et on peut donner une minoration de $\mathscr{S}_C$, tandis que \cite{Wang17} le niveau était $1/2$ et le terme $\mathscr{S}_C$  minoré par 0. C'est pourquoi on peut améliorer le résultat de \cite{Wang17} et obtenir le Corollaire \ref{cor1} suivant. Ces deux changements sont à l'origine de l'amélioration du résultat de \cite{Wang17}, donnée dans le Corollaire \ref{cor1}.

\vskip 8mm

\section{D\'{e}monstration du Corollaire \ref{cor1}}

En prenant $\alpha=1$ (c'est-à-dire, $y=x$) dans le Th\'{e}or\`{e}me \ref{thm2}(ii)
et en remarquant que $\omega(u)=0$ pour $0\leqslant u\leqslant 1$, on obtient
\begin{align}\label{aim:thm2}
\sum_{\substack{n\leqslant x\\ P^+(n)<P^+(n+1)}} 1
& \ge \{h(c) + f(\delta) + o(1)\}x,
\end{align}
o\`{u}
\begin{align*}
h(c)
& := \log\bigg(\frac{1}{1-c}\bigg)- 2\int_{0}^{c}\log\bigg(\frac{1}{1-t}\bigg) \frac{{\rm{d}}t}{\frac{4}{7}-t}
\qquad
(0< c\leqslant \tfrac{2}{7}-\varepsilon),
\\\noalign{\vskip 1mm}
f(\delta)
& := \rho\big(\delta^{-1}\big)\delta\log(2\delta)^{-1}
\qquad
(c\leqslant \delta\leqslant \tfrac{1}{2}).
\end{align*}

\vskip 2mm

Pour $h(c)$, on  fait un calcul similaire à \cite{Wang17}. On a par \emph{Mathematica 9.0}
\begin{align}\label{h(c)}
\max_{0< c\leqslant \frac{2}{7}-\varepsilon }\, h(c)>0,1238
\end{align}
avec $c\approx 0,2056$.

Ainsi, il reste \`{a} trouver $\delta_0\in [0,2056,\, 0,5[$ tel que $f(\delta_0)$ soit proche de la valeur maximale de $f(\delta)$. Au premier abord, on peut calculer facilement que
$$
f'(\tfrac{1}{2})<0
\qquad \text{et}\qquad
f'(\delta)>0
\quad (\delta\leqslant \tfrac{1}{3}).
$$
$\delta_0$ est ainsi dans l'intervalle $(1/3,\, 1/2)$ et on a
\begin{align*}
f(\delta)
= \bigg(1-\log\frac{1}{\delta}+\int_{1}^{1/\delta-1}\frac{\log t}{t+1}\text{d}t\bigg)\delta\log\frac{1}{2\delta}
\qquad (\tfrac{1}{3}\leqslant\delta\leqslant \tfrac{1}{2}).
\end{align*}
Le graphe de $f(\delta)$ sur $[1/3,\, 1/2]$ est le suivant
\begin{figure}[H]
\centerline{
\psfig{figure=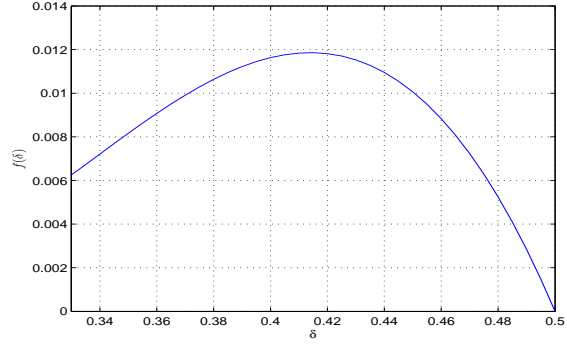,height=5cm,width=8.3cm,angle=0}
}
\caption{$\text{graphe de}\ \delta\rightarrow f(\delta)$ \text{sur} [1/3,\, 1/2]}
\label{fig:f(delta)}
\end{figure}

\`{A} l'aide de \emph{Mathematica 9.0} on peut calculer que
\begin{align}\label{f(delta0)}
\max_{0,2056\leqslant \delta\leqslant 0,5}f(\delta)=f(\delta_0)>0,0118
\end{align}
avec $\delta_0\approx 0,414$.

Compte tenu de \eqref{aim:thm2}, \eqref{h(c)} et \eqref{f(delta0)} on obtient finalement pour $x\rightarrow\infty$
\begin{align*}
\sum_{\substack{n\leqslant x\\P^+(n)<P^+(n+1)}}1  > 0,1238x + 0,0118x=0,1356x.
\end{align*}
Cela ach\`{e}ve la d\'{e}monstration du Corollaire 1.

\vskip 8mm

\section{Démonstration du Corollaire \ref{cor2}}

En remarquant que
$$
P^+(n-1)>P^+(n)<P^+(n+1)
\;\Rightarrow\;
\delta(n)\geqslant 2,
$$
le Th\'{e}or\`{e}me 1 nous permet de déduire que
\begin{align*}
\sum_{n\leqslant x} \delta(n)^{-1}
& = \sum_{\substack{n\leqslant x\\ ``P^+(n-1)>P^+(n)<P^+(n+1)" \text{n'a pas lieu}}} \delta(n)^{-1}
+ \sum_{\substack{n\leqslant x\\ P^+(n-1)>P^+(n)<P^+(n+1)}} \delta(n)^{-1}
\\
& \le \sum_{\substack{n\leqslant x\\ ``P^+(n-1)>P^+(n)<P^+(n+1)" \text{n'a pas lieu}}} 1
+ \sum_{\substack{n\leqslant x\\ P^+(n-1)>P^+(n)<P^+(n+1)}} \frac{1}{2}
\\
& = x - \sum_{\substack{n\leqslant x\\ P^+(n-1)>P^+(n)<P^+(n+1)}} \frac{1}{2}
\\\noalign{\vskip 1mm}
& \le (1-5,315\times 10^{-8})x .
\end{align*}

On peut obtenir une meilleur majoration en utilisant le Théorème \ref{thm3}. Pour $j\geqslant 1,$
on note $C_j$ l'ensemble des {\it creux} d'ordre $j$:
\begin{align*}
C_j:=C_j(x)=\big\{n\leqslant x:\ P^+(n)<P^+(n+k)\ \text{pour tous}\ 1\leqslant |k|\leqslant j \big\},
\end{align*}
et pour $C_0,$ on prend $C_0=(C_1)^c,$ le complémentaire de $C_1.$ Alors on a pour tout entier $J\geqslant 1$
\begin{align*}
\sum_{n\leqslant x} \delta(n)^{-1}
& = \sum_{n\in C_0} \delta(n)^{-1}+ \sum_{n\in C_1} \delta(n)^{-1}
\\
& = \sum_{n\in C_0} \delta(n)^{-1}+ \sum_{n\in C_1\backslash C_2} \delta(n)^{-1}+\sum_{n\in C_2} \delta(n)^{-1}
\\
& = \sum_{n\in C_0} \delta(n)^{-1}+ \sum_{j=2}^{J}\ \sum_{n\in C_{j-1}\backslash C_j} \delta(n)^{-1}
+\sum_{n\in C_J} \delta(n)^{-1}
\\
& \leqslant \sum_{n\not\in C_1}1+ \sum_{j=2}^{J}\, \frac{1}{j}\big(|C_{j-1}|-|C_{j}|\big)+\frac{C_{J}}{J+1}
\\
& \leqslant x- \sum_{j=1}^{J}|C_{j}| \Big(\frac{1}{j}-\frac{1}{j+1}\Big).
\end{align*}
Puis, on applique la minoration de $|C_{j}|$, $j\geqslant 2$ du Théorème \ref{thm3} et
la minoration de $|C_{1}|$ du Théorème \ref{thm1}, on obtient
\begin{align}\label{cor2-thm3}
\sum_{n\leqslant x} \delta(n)^{-1}
\leqslant x\Big\{1-5,315\times 10^{-8}- \sum_{j=2}^{J}C_3(j) \Big(\frac{1}{j}-\frac{1}{j+1}\Big)\Big\},
\end{align}
où $C_3(j)$ est une constante strictement positive  définie par \eqref{C3(J)}.
On peut donc obtenir une meilleure majoration pour $\delta(n)$.
Mais le gain est minime, on ne calcule pas ici la valeur numérique.

La démonstration du Corollaire $2^*$ est analogue. De plus, en utilisant le Théorème \ref{thm4}, on peut aussi obtenir un meilleur résultat pour $\delta_*(n)$ comme \eqref{cor2-thm3} ci-dessus .
On omet ici les détails.

\vskip 8mm

\section{Démonstration du Corollaire \ref{cor3}}
On note pour $x\rightarrow \infty$
\begin{align*}
a_1(x):=\big|\big\{n\leqslant x:\, P^+(n-1)>P^+(n)<P^+(n+1)\big\}\big|,
\\\noalign{\vskip 1mm}
a_2(x):=\big|\big\{n\leqslant x:\, P^+(n-1)<P^+(n)>P^+(n+1)\big\}\big|,
\\\noalign{\vskip 1mm}
a_3(x):=\big|\big\{n\leqslant x:\, P^+(n-1)<P^+(n)<P^+(n+1)\big\}\big|,
\\\noalign{\vskip 1mm}
a_4(x):=\big|\big\{n\leqslant x:\, P^+(n-1)>P^+(n)>P^+(n+1)\big\}\big|,
\end{align*}
o\`{u} $0\leqslant a_1(x), a_2(x), a_3(x), a_4(x) \leqslant [x],\ a_1(x)+ a_2(x)+ a_3(x)+ a_4(x)=[x].$
\vskip 1mm

On a tout d'abord, d'apr\`{e}s \eqref{eq:cor1} du Corollaire \ref{cor1}
\begin{align*}
[x]-\big(a_2(x)+a_3(x)\big)=a_1(x)+a_4(x)>0,1356x,
\nonumber\\\noalign{\vskip 2mm}
[x]-\big(a_2(x)+a_4(x)\big)=a_1(x)+a_3(x)>0,1356x.
\end{align*}
La minoration de $a_2(x)$ dans \eqref{eq:thm1_B} du Théorème \ref{thm1} permet d'obtenir
\begin{align*}
a_3(x)< (0,8644-8,84\times 10^{-4})x,
\qquad
a_4(x)< (0,8644-8,84\times 10^{-4})x.
\end{align*}
Pour $a_1(x)$, on d\'{e}duit, en utilisant la minoration de \eqref{mino-delta(n)}
\begin{equation}\label{majo-a1}
\begin{aligned}
\frac{2}{3}x+o(x)
& <\sum_{n\leqslant x}\frac{1}{\delta(n)}
\\\noalign{\vskip 1mm}
& \leqslant a_2(x)+a_3(x)+a_4(x)+\frac{a_1(x)}{2}
\\\noalign{\vskip 1mm}
& \leqslant x-a_1(x)+\frac{a_1(x)}{2}.
\end{aligned}
\end{equation}
Par suite, \eqref{majo-a1} entra\^{i}nent que
\begin{align*}
a_1(x) <\frac{2}{3}x.
\end{align*}
Finalement, similaire \`{a} \eqref{majo-a1}, on peut obtenir
\begin{align*}
a_2(x)< \frac{2}{3}x
\end{align*}
en utilisant la minoration de \eqref{mino-delta*(n)}
\begin{align*}
\sum_{n\leqslant x}\frac{1}{\delta_*(n)} > \frac{2}{3}x +o(x).
\end{align*}
Ce qui termine la démonstration du Corollaire 3.

\vskip 7mm

\bibliographystyle{plain}
\bibliography{SurLesPlusGrandsFacteursPremierDesEntiersConsecutifs}

\end{document}